\documentclass[12pt,a4paper]{amsart}
\usepackage{amssymb}

\calclayout

\newcommand{\+}{\protect\nobreakdash-}

\renewcommand{\:}{\colon}

\newcommand{\rarrow}{\longrightarrow}

\newcommand{\ot}{\otimes}

\newcommand{\lrarrow}{\mskip.5\thinmuskip\relbar\joinrel\relbar\joinrel
 \rightarrow\mskip.5\thinmuskip\relax}

\newcommand{\bu}{{\text{\smaller\smaller$\scriptstyle\bullet$}}}

\DeclareMathOperator{\Hom}{Hom}
\DeclareMathOperator{\Ext}{Ext}
\DeclareMathOperator{\Tor}{Tor}

\DeclareMathOperator{\id}{id}

\DeclareMathOperator{\cone}{cone}

\newcommand{\Modl}{{\operatorname{\mathsf{--Mod}}}}
\newcommand{\Modr}{{\operatorname{\mathsf{Mod--}}}}
\newcommand{\Modrfl}{{\operatorname{\mathsf{Mod_{fl}--}}}}
\newcommand{\Modrproj}{{\operatorname{\mathsf{Mod_{proj}--}}}}
\newcommand{\Modrprojfp}{{\operatorname{\mathsf{Mod_{proj}^{fp}--}}}}

\newcommand{\Sets}{\mathsf{Sets}}
\newcommand{\Ab}{\mathsf{Ab}}
\newcommand{\Funct}{\mathsf{Funct}}
\newcommand{\Hot}{\mathsf{Hot}}
\newcommand{\Ac}{\mathsf{Ac}}
\newcommand{\Fil}{\mathsf{Fil}}
\newcommand{\PInj}{\mathsf{PInj}}
\newcommand{\PProj}{\mathsf{PProj}}
\newcommand{\Inj}{\mathsf{Inj}}
\newcommand{\Proj}{\mathsf{Proj}}
\newcommand{\Cot}{\mathsf{Cot}}
\newcommand{\FpProj}{\mathsf{FpProj}}
\newcommand{\inj}{\mathsf{inj}}
\newcommand{\proj}{\mathsf{proj}}

\newcommand{\fp}{\mathsf{fp}}
\newcommand{\sfp}{{\operatorname{\mathsf{s-fp}}}}
\newcommand{\wfp}{{\operatorname{\mathsf{w-fp}}}}
\newcommand{\pur}{\mathsf{pur}}
\newcommand{\ad}{\mathsf{ad}}

\newcommand{\sop}{{\mathsf{op}}}

\newcommand{\cL}{\mathcal L}
\newcommand{\cR}{\mathcal R}

\newcommand{\sA}{\mathsf A}
\newcommand{\sB}{\mathsf B} 
\newcommand{\sC}{\mathsf C}
\newcommand{\sD}{\mathsf D}
\newcommand{\sE}{\mathsf E}
\newcommand{\sF}{\mathsf F}

\newcommand{\sK}{\mathsf K}

\newcommand{\sP}{\mathsf P}

\newcommand{\sS}{\mathsf S}
\newcommand{\sT}{\mathsf T}

\newcommand{\boZ}{\mathbb Z}
\newcommand{\boQ}{\mathbb Q}

\newcommand{\Section}[1]{\bigskip\section{#1}\medskip}
\theoremstyle{plain}
\newtheorem{thm}{Theorem}[section]

\newtheorem{lem}[thm]{Lemma}
\newtheorem{prop}[thm]{Proposition}
\newtheorem{cor}[thm]{Corollary}
\theoremstyle{definition}
\newtheorem{ex}[thm]{Example}

\newtheorem{rem}[thm]{Remark}

\begin{document}

\title{Fp-projective periodicity}

\author[S.~Bazzoni]{Silvana Bazzoni}

\address[Silvana Bazzoni]{%
Dipartimento di Matematica ``Tullio Levi-Civita'' \\
Universit\`a di Padova \\
Via Trieste 63, 35121 Padova \\ Italy}

\email{bazzoni@math.unipd.it}

\author[M.~Hrbek]{Michal Hrbek}

\address[Michal Hrbek]{Institute of Mathematics, Czech Academy
of Sciences, \v Zitn\'a 25, 115\,67 Prague~1, Czech Republic}

\email{hrbek@math.cas.cz}

\author[L.~Positselski]{Leonid Positselski}

\address[Leonid Positselski]{%
Institute of Mathematics, Czech Academy of Sciences \\
\v Zitn\'a~25, 115~67 Praha~1, Czech Republic}

\email{positselski@math.cas.cz}

\begin{abstract}
 The phenomenon of periodicity, discovered by Benson and Goodearl,
is linked to the behavior of the objects of cocycles in
acyclic complexes.
 It is known that any flat $\Proj$\+periodic module is projective,
any fp\+injective $\Inj$\+periodic module is injective, and any
$\Cot$\+periodic module is cotorsion.
 It is also known that any pure $\PProj$\+periodic module is
pure-projective and any pure $\PInj$\+periodic module is pure-injective.
 Generalizing a result of \v Saroch and \v St\!'ov\'\i\v cek, we show
that every $\FpProj$\+periodic module is weakly fp\+projective.
 The proof is quite elementary, using only a strong form of
the pure-projective periodicity and the Hill lemma.
 More generally, we prove that, in a locally finitely presentable
Grothendieck category, every $\FpProj$\+periodic object is weakly
fp\+projective.
 In a locally coherent category, all weakly fp\+projective objects are
fp\+projective.
 We also present counterexamples showing that a non-pure
$\PProj$\+periodic module over a regular finitely generated commutative
algebra (or a hereditary finite-dimensional associative algebra)
over a field need not be pure-projective.
\end{abstract}

\maketitle

\tableofcontents

\section*{Introduction}
\medskip

\setcounter{subsection}{-1}
\subsection{{}} \label{introd-periodic-defined-subsecn}
 Any acyclic bounded above complex of projective modules is
contractible.
 The same applies to acyclic bounded below complexes of injective
modules.
 For unbounded\- complexes, these assertions no longer hold.

 Here is the classical counterexample.
 Let $k$~be a field, and let $\Lambda=k[\epsilon]/(\epsilon^2)$ be
the ring of dual numbers, or which is the same, the exterior algebra
with one generator over~$k$.
 Consider the acyclic, unbounded complex of projective-injective
$\Lambda$\+modules
\begin{equation} \label{unbounded-acyclic-classical}
 \dotsb\lrarrow \Lambda\overset{\epsilon*}\lrarrow\Lambda
 \overset{\epsilon*}\lrarrow\Lambda \lrarrow\dotsb
\end{equation}
 Here every term of the complex~\eqref{unbounded-acyclic-classical}
is the free $\Lambda$\+module with one generator; all the differentials
in~\eqref{unbounded-acyclic-classical} are the operators of
multiplication with~$\epsilon$.
 One can easily see that the complex~\eqref{unbounded-acyclic-classical}
is not contractible.

 From the homological algebra perspective, \emph{periodicity theorems}
constitute one of the several presently known technical approaches
designed to overcome the difficulties represented by complexes such
as~\eqref{unbounded-acyclic-classical}.
 Other approaches include such concepts as
the \emph{homotopy projective} and \emph{homotopy injective}
complexes~\cite{Spal}, the \emph{coderived} and
\emph{contraderived categories}~\cite[Section~7]{Pksurv}, etc.

 Let $\sK$ be an abelian category, and let $\sA\subset\sK$ be a class
of objects.
 An object $M\in\sK$ is said to be \emph{$\sA$\+periodic} if there
exists a short exact sequence $0\rarrow M\rarrow A\rarrow M\rarrow0$
in $\sK$ with $A\in\sA$.
 For example, the acyclic complex~\eqref{unbounded-acyclic-classical}
demonstrates the fact that the irreducible $\Lambda$\+module~$k$ is
periodic with respect to the class of all projective-injective
$\Lambda$\+modules, while \emph{not} belonging to this class.

 More generally, one can let $\sK$ be an exact category (in Quillen's
sense), and consider conflations instead of the short exact sequences.
 A typical \emph{periodicity theorem} has the form: for some two
classes of objects $\sA\subset\sE\subset\sK$, any $\sA$\+periodic
object belonging to $\sE$ belongs to~$\sA$.

\subsection{{}} \label{introd-module-periodicity-theorems-subsecn}
 The subject of periodicity originates from the seminal paper of
Benson and Goodearl~\cite{BG}, where the following theorem was proved.

\begin{thm}[Benson and Goodearl~{\cite[Theorem~2.5]{BG}}]
\label{benson-goodearl-flat-projective}
 Let $R$ be a ring; denote by\/ $\Proj$ the class of all projective
$R$\+modules.
 Then any flat\/ $\Proj$\+periodic $R$\+module is projective.
\end{thm}

 The theorem of Benson and Goodearl had module-theoretic flavor.
 It was rediscovered and strengthened by Neeman~\cite{Neem} in the form
of the following theorem in homological algebra.

\begin{thm}[Neeman~{\cite[Remark~2.15 and Theorem~8.6]{Neem}};
see also Theorem~\ref{neeman-theorem} below]
\label{neeman-flat-projective}
 Let $R$ be a ring.  Then \par
\textup{(a)} any acyclic complex of projective $R$\+modules with flat
modules of cocycles is contractible; \par
\textup{(b)} moreover, if $P^\bu$ is a complex of projective
$R$\+modules and $F^\bu$ is an acyclic complex of flat $R$\+modules
with flat modules of cocycles, then any morphism of complexes
$P^\bu\rarrow F^\bu$ is homotopic to zero.
\end{thm}

 Part~(a) of Theorem~\ref{neeman-flat-projective} is obviously
a particular case of part~(b) (take $F^\bu=P^\bu$ and consider
the identity morphism $P^\bu\rarrow F^\bu$).
 Theorem~\ref{neeman-flat-projective}(a) is an equivalent restatement
of Theorem~\ref{benson-goodearl-flat-projective}.

 Indeed, for any abelian category $\sK$ with exact functor of countable
coproduct (respectively, product), any class of objects $\sA\subset
\sK$ closed under countable coproducts (resp., products), and any
class of objects $\sB\subset\sK$ closed under direct summands,
the following two conditions are equivalent:
\begin{itemize}
\item every $\sA$\+periodic object in $\sK$ belongs to~$\sB$;
\item in any acyclic complex in $\sK$ with the terms belonging to
$\sA$, the objects of cocycles belong to~$\sB$.
\end{itemize}
 This elementary observation can be found in~\cite[proof of
Proposition~7.6]{CH} or~\cite[Propositions~1 and~2]{EFI}.

 An $R$\+module $J$ is said to be \emph{fp\+injective} if
$\Ext^1_R(T,J)=0$ for all finitely presented $R$\+modules~$T$.
 Fp\+injective modules are also known as \emph{absolutely pure} modules.  
 They are often considered as dual analogues of flat modules.
 Thus the following theorem due to \v St\!'ov\'\i\v cek~\cite{Sto} is
dual-analogous to Theorem~\ref{benson-goodearl-flat-projective}.
 Another proof can be found in the paper of Bazzoni, Cort\'es-Izurdiaga,
and Estrada~\cite{BCE}.

\begin{thm}[essentially
\v St\!'ov\'\i\v cek~{\cite[Corollary~5.5]{Sto}};
see also~{\cite[Theorem~1.2(1) or Proposition~4.8(1)]{BCE}}]
\label{stovicek-fp-injective-inj-periodic-module}
 Let $R$ be a ring; denote by\/ $\Inj$ the class of all injective
$R$\+modules.
 Then any fp\+injective\/ $\Inj$\+periodic $R$\+module is injective.
\end{thm}

 The following homological algebra theorem provides a reformulation
and a stronger version of
Theorem~\ref{stovicek-fp-injective-inj-periodic-module}.

\begin{thm}[essentially \v St\!'ov\'\i\v cek~{\cite[Theorem~5.4 and
Corollary~5.5]{Sto}}; see also~{\cite[Theorem~5.1(1)]{BCE}}
for part~(a)]
\label{stovicek-fp-injective-injective-modules}
 Let $R$ be a ring.  Then \par
\textup{(a)} any acyclic complex of injective $R$\+modules with
fp\+injective modules of cocycles is contractible; \par
\textup{(b)} moreover, if $J^\bu$ is a complex of injective $R$\+modules
and $I^\bu$ is an acyclic complex of fp\+injective $R$\+modules with
fp\+injective modules of cocycles, then any morphism of complexes
$I^\bu\rarrow J^\bu$ is homotopic to zero.
\end{thm}

 Part~(a) of Theorem~\ref{stovicek-fp-injective-injective-modules} is
obviously a particular case of part~(b).
 Theorem~\ref{stovicek-fp-injective-injective-modules}(a) is
an equivalent restatement of
Theorem~\ref{stovicek-fp-injective-inj-periodic-module}.

 An $R$\+module $C$ is said to be \emph{cotorsion} if $\Ext^1_R(F,C)=0$
for all flat $R$\+modules~$F$.
 Thus the following theorem due to Bazzoni, Cort\'es-Izurdiaga,
and Estrada~\cite{BCE} can be viewed as complementing
Theorem~\ref{benson-goodearl-flat-projective}.

\begin{thm}[Bazzoni, Cort\'es-Izurdiaga, and
Estrada~{\cite[Theorem~1.2(2) or Proposition~4.8(2)]{BCE}}]
\label{bazzoni-cortes-estrada-cotorsion-periodic}
 Let $R$ be a ring; denote by\/ $\Cot$ the class of all cotorsion
$R$\+modules.
 Then any\/ $\Cot$\+periodic $R$\+module is cotorsion.
\end{thm}

 The following homological algebra theorem is a reformulation of
Theorem~\ref{bazzoni-cortes-estrada-cotorsion-periodic}.

\begin{thm}[Bazzoni, Cort\'es-Izurdiaga, and
Estrada~{\cite[Theorems~5.1(2) and~5.3]{BCE}}]
\label{bazzoni-cortes-estrada-cotorsion-flat-modules}
 Let $R$ be a ring.  Then \par
\textup{(a)} any acyclic complex of cotorsion $R$\+modules has cotorsion
modules of cocycles; \par
\textup{(b)} if $C^\bu$ is a complex of cotorsion $R$\+modules
and $F^\bu$ is an acyclic complex of flat $R$\+modules with flat
modules of cocycles, then any morphism of complexes
$F^\bu\rarrow C^\bu$ is homotopic to zero.
\end{thm}

 Parts~(a) and~(b) of
Theorem~\ref{bazzoni-cortes-estrada-cotorsion-flat-modules} can be
easily deduced from each other.
 In the language of~\cite[Definition~3.3]{Gil},
Theorem~\ref{bazzoni-cortes-estrada-cotorsion-flat-modules}(b) tells
that any complex of cotorsion modules is dg\+cotorsion.

 The pair of classes (all flat modules, all cotorsion modules) is
a thematic example of what people call a \emph{hereditary complete
cotorsion pair} in the category of $R$\+modules~\cite[Section~5.2
and Chapter~6]{GT}.
 In this sense, it is tempting to try and deduce
Theorem~\ref{bazzoni-cortes-estrada-cotorsion-periodic} from
Theorem~\ref{benson-goodearl-flat-projective}, or
Theorem~\ref{bazzoni-cortes-estrada-cotorsion-flat-modules}
from Theorem~\ref{neeman-flat-projective}, or vice versa.
 However, this does not seem to work well.
 The theorem on $\Proj$\+periodic flat modules and the theorem on
$\Cot$\+periodic modules appear to be two quite different results
nicely complementing each other.

 The aim of this paper is to obtain a similarly complementary result
to Theorems~\ref{stovicek-fp-injective-inj-periodic-module}
and~\ref{stovicek-fp-injective-injective-modules}.
 An $R$\+module $P$ is said to be
\emph{fp\+projective}~\cite[Definition~3.3 and Theorem~3.4(2)]{Trl},
\cite{MD} if $\Ext^1_R(P,J)=0$ for all fp\+injective $R$\+modules~$J$.
 The pair of classes (fp\+projective modules, fp\+injective modules)
is a complete cotorsion pair; this cotorsion pair in the category
of right $R$\+modules is hereditary if and only if the ring $R$
is right coherent.
 The potential importance of fp\+injective and fp\+projective modules
in the context of semi-infinite homological algebra and algebraic
geometry over coherent rings or schemes was emphasized
in the paper~\cite{Pfp}.

 To remedy the failure of the (fp\+projective, fp\+injective) cotorsion
pair to be hereditary over a noncoherent ring, we follow
the suggestion of~\cite[Section~4]{EK} and consider \emph{strongly
fp\+injective modules}.
 An $R$\+module $J$ is said to be \emph{strongly fp\+injective} if
$\Ext^n_R(T,J)=0$ for all finitely presented $R$\+modules $T$ and
$n\ge1$.
 We say that an $R$\+module $P$ is \emph{weakly fp\+projective} if
$\Ext^1_R(P,J)=0$ for all strongly fp\+injective $R$\+modules~$J$. 
 The pair of classes (weakly fp\+projective modules, strongly
fp\+injective modules) is a hereditary complete cotorsion pair over
any ring~$R$.
 The classes of fp\+projective and weakly fp\+projective right
$R$\+modules coincide if and only if the ring $R$ is right coherent.

 The following theorem is the main result of this paper formulated in
the module-theoretic language.

\begin{thm}[Corollary~\ref{modules-weakly-fp-periodicity-cor} below]
\label{fp-proj-periodic-module}
 Let $R$ be a ring; denote by\/ $\FpProj$ the class of all
fp\+projective $R$\+modules.
 Then any $\FpProj$\+periodic $R$\+module is weakly fp\+projective.
\end{thm}

 The next theorem provides a formulation of our main result in
the language of homological algebra of complexes of modules.

\begin{thm}[Corollaries~\ref{modules-cycles-in-fp-projective-cor}
and~\ref{modules-hot-hom-fp-proj-fp-inj-cor} below]
\label{fp-proj-fp-inj-modules}
 Let $R$ be a ring.  Then \par
\textup{(a)} any acyclic complex of fp\+projective $R$\+modules has
weakly fp\+projective modules of cocycles; \par
\textup{(b)} if $P^\bu$ is a complex of fp\+projective $R$\+modules
and $J^\bu$ is an acyclic complex of fp\+injective $R$\+modules with
fp\+injective modules of cocycles, then any morphism of complexes
$P^\bu\rarrow J^\bu$ is homotopic to zero.
\end{thm}

 In the preceding paper of \v Saroch and
\v St\!'ov\'\i\v cek~\cite[Example~4.3]{SarSt}, the results of our
Theorems~\ref{fp-proj-periodic-module}\+-\ref{fp-proj-fp-inj-modules}
were obtained for right modules over a right coherent ring $R$
using complicated set-theoretic techniques.
 Our proof in the present paper is both more elementary and provides
a more general version of fp\+projective periodicity, in that we
do not assume coherence.
 On the other hand, it is easy to produce a counterexample showing
that, over any ring that is not right coherent, a non-fp-projective
(but weakly fp\+projective!) $\FpProj$\+periodic right module exists.

 To the extent that fp\+injective modules can be viewed as dual
analogues of flat modules, one can also view fp\+projective modules
as dual analogues of cotorsion modules.
 Thus our Theorems~\ref{fp-proj-periodic-module}
and~\ref{fp-proj-fp-inj-modules} are dual-analogous to
Theorems~\ref{bazzoni-cortes-estrada-cotorsion-periodic}
and~\ref{bazzoni-cortes-estrada-cotorsion-flat-modules} of
Bazzoni, Cort\'es-Izurdiaga, and Estrada.

 Theorem~\ref{fp-proj-fp-inj-modules}(a) is
an equivalent restatement
of Theorem~\ref{fp-proj-periodic-module}.
 In the language of~\cite[Definition~3.3]{Gil},
Theorem~\ref{fp-proj-fp-inj-modules}(b) tells that any
complex of fp\+projective modules is dg\+fp\+projective.
 Assuming the ring to be coherent, \v Saroch and \v St\!'ov\'\i\v cek
in~\cite[Example~4.3]{SarSt} deduced what is stated above as
Theorem~\ref{fp-proj-fp-inj-modules}(b) from
Theorem~\ref{fp-proj-fp-inj-modules}(a).
 Our proof of Theorem~\ref{fp-proj-fp-inj-modules}(a) deduces it from
Theorem~\ref{fp-proj-fp-inj-modules}(b), while the latter, in turn,
is obtained by reducing the problem to a \emph{pure}
version of Theorem~\ref{neeman-flat-projective}(b).

 As pure periodicity theorems play an important role in our proofs and
are also interesting on their own, we will formulate them below in this
introduction, firstly in the context of module categories and then
more generally.

\subsection{{}} \label{introd-pure-periodicity}
 Recall that a short exact sequence of modules is called \emph{pure} if
its exactness is preserved by taking the tensor product with any module,
or equivalently, by taking the Hom from any finitely presented
module~\cite[Definition~2.6 and Lemma~2.19]{GT}.
 The class of all pure exact sequences defines an exact category
structure on the category of $R$\+modules, called the \emph{pure
exact structure}.
 The acyclic complexes, projective objects, and injective objects with
respect to this exact structure are called \emph{pure acyclic},
\emph{pure-projective}, and \emph{pure-injective}, respectively.
 Given a class of $R$\+modules $\sA$, an $R$\+module $M$ is said to be
\emph{pure\/ $\sA$\+periodic} if there exists a pure short exact
sequence $0\rarrow M\rarrow A\rarrow M\rarrow0$ with $A\in\sA$.

 The following theorem due to Simson~\cite{Sim} is a pure version of
Theorem~\ref{benson-goodearl-flat-projective}.

\begin{thm}[Simson~{\cite[Theorem~1.3 or~4.4]{Sim}}]
\label{simson-pure-proj-periodic}
 Let $R$ be a ring; denote by\/ $\PProj$ the class of all
pure-projective $R$\+modules.
 Then any pure\/ $\PProj$\+periodic $R$\+module is pure-projective.
\end{thm}

 Theorem~\ref{simson-pure-proj-periodic} is a module-theoretic
formulation of pure-projective periodicity.
 The following theorem of \v St\!'ov\'\i\v cek~\cite{Sto} provides
a homological formulation, which is a pure version of
Theorem~\ref{neeman-flat-projective}.

\begin{thm}[\v St\!'ov\'\i\v cek~{\cite[Theorem~5.4 and
Corollary~5.5]{Sto}}]
\label{stovicek-pure-projective-modules}
 Let $R$ be a ring.  Then \par
\textup{(a)} any pure acyclic complex of pure-projective $R$\+modules
is contractible; \par
\textup{(b)} moreover, if $P^\bu$ is a complex of pure-projective
$R$\+modules and $X^\bu$ is a pure acyclic complex of $R$\+modules, then
any morphism of complexes $P^\bu\rarrow X^\bu$ is homotopic to zero.
\end{thm}

 Part~(a) of Theorem~\ref{stovicek-pure-projective-modules} is
obviously a particular case of part~(b).
 Theorem~\ref{stovicek-pure-projective-modules}(a) is an equivalent 
restatement of Theorem~\ref{simson-pure-proj-periodic}.

 Theorems~\ref{simson-pure-proj-periodic}
and~\ref{stovicek-pure-projective-modules} are essentially equivalent
to Theorems~\ref{benson-goodearl-flat-projective}
and~\ref{neeman-flat-projective}, respectively.
 To deduce Theorems~\ref{benson-goodearl-flat-projective}
and~\ref{neeman-flat-projective} from
Theorems~\ref{simson-pure-proj-periodic}
and~\ref{stovicek-pure-projective-modules}, it suffices to observe that
any short exact sequence of flat modules is pure, any projective module
is pure-projective, and any flat pure-projective module is projective.

 To deduce Theorems~\ref{simson-pure-proj-periodic}
and~\ref{stovicek-pure-projective-modules} from
Theorems~\ref{benson-goodearl-flat-projective}
and~\ref{neeman-flat-projective}, one has to interpret the essentially
small additive category of finitely presented right $R$\+modules as
a ``ring with many objects''~$\cR$.
 Then pure-projective right $R$\+modules are the same things as
projective right $\cR$\+modules, while arbitrary right $R$\+modules are
the same things as flat right $\cR$\+modules, and pure exact sequences
of right $R$\+modules correspond to exact sequences of flat right
$\cR$\+modules.
 Applying Theorems~\ref{benson-goodearl-flat-projective}
and~\ref{neeman-flat-projective} to right $\cR$\+modules yields
Theorems~\ref{simson-pure-proj-periodic}
and~\ref{stovicek-pure-projective-modules} for right $R$\+modules
(respectively).

 The following theorem of \v St\!'ov\'\i\v cek~\cite{Sto} is a pure
version of Theorem~\ref{stovicek-fp-injective-inj-periodic-module}.

\begin{thm}[\v St\!'ov\'\i\v cek~{\cite[Corollary~5.5]{Sto}}]
\label{stovicek-pure-inj-periodic-module}
 Let $R$ be a ring; denote by\/ $\PInj$ the class of all
pure-injective $R$\+modules.
 Then any pure\/ $\PInj$\+periodic $R$\+module is pure-injective.
\end{thm}

 Theorem~\ref{stovicek-pure-inj-periodic-module} is a module-theoretic
formulation of pure-injective periodicity.
 The following theorem from~\cite{Sto} provides a homological
formulation, which is a pure version of
Theorem~\ref{stovicek-fp-injective-injective-modules}.

\begin{thm}[\v St\!'ov\'\i\v cek~{\cite[Theorem~5.4 and
Corollary~5.5]{Sto}}]
\label{stovicek-pure-injective-modules}
 Let $R$ be a ring.  Then \par
\textup{(a)} any pure acyclic complex of pure-injective $R$\+modules
is contractible; \par
\textup{(b)} moreover, if $J^\bu$ is a complex of pure-injective
$R$\+modules and $X^\bu$ is a pure acyclic complex of $R$\+modules, then
any morphism of complexes $X^\bu\rarrow J^\bu$ is homotopic to zero.
\end{thm}

 Part~(a) of Theorem~\ref{stovicek-pure-injective-modules} is obviously
a particular case of part~(b).
 Theorem~\ref{stovicek-pure-injective-modules}(a) is an equivalent
restatement of Theorem~\ref{stovicek-pure-inj-periodic-module}.

 Theorems~\ref{stovicek-pure-inj-periodic-module}
and~\ref{stovicek-pure-injective-modules} are essentially equivalent to
Theorems~\ref{stovicek-fp-injective-inj-periodic-module}
and~\ref{stovicek-fp-injective-injective-modules}, respectively.
 To deduce Theorems~\ref{stovicek-fp-injective-inj-periodic-module}
and~\ref{stovicek-fp-injective-injective-modules} from
Theorems~\ref{stovicek-pure-inj-periodic-module}
and~\ref{stovicek-pure-injective-modules}, it suffices to observe that
any short exact sequence of fp\+injective modules is pure, any
injective module is pure-injective, and any fp\+injective
pure-injective module is injective.

 To deduce Theorems~\ref{stovicek-pure-inj-periodic-module}
and~\ref{stovicek-pure-injective-modules} from
Theorems~\ref{stovicek-fp-injective-inj-periodic-module}
and~\ref{stovicek-fp-injective-injective-modules}, one has to
interpret the essentially small additive category of finitely presented
left $R$\+modules as a ``ring with many objects''~$\cL$.
 Then pure-injective right $R$\+modules are the same things as
injective left $\cL$\+modules, while arbitrary right $R$\+modules
are the same things as fp\+injective left $\cL$\+modules, and
pure exact sequences of right $R$\+modules correspond to exact sequences
of fp\+injective left $\cL$\+modules.
 Applying Theorems~\ref{stovicek-fp-injective-inj-periodic-module}
and~\ref{stovicek-fp-injective-injective-modules} to left $\cL$\+modules
yields Theorems~\ref{stovicek-pure-inj-periodic-module}
and~\ref{stovicek-pure-injective-modules} for right $R$\+modules
(respectively).

 The assertions about left $\cL$\+modules from the previous paragraph
are more counterintuitive than the discussion of right $\cR$\+modules
above, so let us explain them in more detail.
 By the definition, the left $\cL$\+modules are the covariant
functors to the category of abelian groups $\cL\rarrow\Ab$.
 To a right $R$\+module $M$, the tensor product functor $S\longmapsto
M\ot_R S\:\cL\rarrow\Ab$ is assigned.
 Clearly, this identifies the category of right $R$\+modules
$\Modr R$ with the category of right exact functors $\cL\rarrow\Ab$.
 The claim is that a functor $\cL\rarrow\Ab$ is right exact if and
only if it is fp\+injective as an object of the functor category
$\cL\Modl$.

 Indeed, let $G\:\cL\rarrow\Ab$ be a finitely presented left
$\cL$\+module.
 Then $G$ is the cokernel of a morphism of representable functors
$\Hom_R(S,{-})\rarrow\Hom_R(T,{-})$, where $S$, $T\in\cL$.
 The morphism of representable functors is induced by a morphism of
finitely presented left $R$\+modules $T\rarrow S$.
 So we have a $4$\+term exact sequence $0\rarrow\Hom_R(S/T,{-})\rarrow
\Hom_R(S,{-})\rarrow\Hom_R(T,{-})\rarrow G\rarrow0$ of functors
$\cL\rarrow\Ab$.
 Here $S/T$ is a notation for the cokernel of the morphism $T\rarrow S$,
which is also a finitely presented left $R$\+module.
 Now it is clear that the ``ring with many objects'' $\cL$ is left
coherent of weak global dimension at most~$2$.
 Hence a left $\cL$\+module $F$ is fp\+injective if and only if
$\Ext^n_\cL(G,F)=0$ for all finitely presented left $\cL$\+modules $G$
and $n\ge1$.
 Using the projective resolution above, the Ext groups in question
are computed by the complex $F(T)\rarrow F(S)\rarrow F(S/T)\rarrow0$
(since $\Hom_\cL(\Hom_R(S,{-}),F)=F(S)$ and similarly for $T$
and~$S/T$),  so they vanish for all $S$, $T\in\cL$ if and only if
the functor $F\:\cL\rarrow\Ab$ is right exact.

 To show that the same correspondence identifies pure-injective right
$R$\+modules with injective left $\cL$\+modules, notice that
the injective left $\cL$\+modules are the direct summands of
the products of the functors $S\longmapsto\Hom_R(S,T)^+\simeq
S\ot_R T^+$, where $M^+$ denotes the character group/module
$M^+=\Hom_\boZ(M,\boQ/\boZ)$ and $S$, $T\in\cL$.
 On the other hand, the pure-injective right $R$\+modules are
the direct summands of the products of the character modules $T^+$
of finitely presented left $R$\+modules~$T$
\,\cite[Corollary~2.30(a)]{GT}.

\subsection{{}}
 The main results of this paper are formulated and proved in
the category-theoretic setting.
 Let $\sK$ be a locally finitely presentable abelian category (any
such category is Grothendieck).
 Just as for modules, an object $J\in\sK$ is called \emph{fp\+injective}
if $\Ext^1_\sK(T,J)=0$ for all finitely presentable objects $T\in\sK$.
 An object $P\in\sK$ is called \emph{fp\+projective} if
$\Ext^1_\sK(P,J)=0$ for all fp\+injective objects $J\in\sK$.

 Furthermore, an object $J\in\sK$ is said to be \emph{strongly
fp\+injective} if $\Ext^n_\sK(T,J)=0$ for all finitely presentable
objects $T\in\sK$ and all integers $n\ge1$.
 An object $P\in\sK$ is said to be \emph{weakly fp\+projective} if
$\Ext^1_\sK(P,J)=0$ for all strongly fp\+injective objects $J\in\sK$.
 The following two theorems are our main results.

\begin{thm}[Corollary~\ref{lfp-category-weakly-fp-periodicity-cor}
below] \label{fp-proj-periodic-object}
 Let\/ $\sK$ be a locally finitely presentable abelian category;
denote by\/ $\FpProj$ the class of all fp\+projective objects
in\/~$\sK$.
 Then every\/ $\FpProj$\+periodic object in\/ $\sK$ is weakly
fp\+projective.
\end{thm}

\begin{thm}[Theorems~\ref{lfp-category-cycles-in-fp-projective-thm}
and~\ref{lfp-category-hot-hom-fp-proj-fp-inj-thm} below]
\label{fp-proj-fp-inj-objects}
 Let\/ $\sK$ be a locally finitely presentable abelian category.
Then \par
\textup{(a)} every acyclic complex of fp\+projective objects in\/ $\sK$
has weakly fp\+projective objects of cocycles; \par
\textup{(b)} if $P^\bu$ is a complex of fp\+projective objects in\/
$\sK$ and $J^\bu$ is an acyclic complex of fp\+injective objects with
fp\+injective objects of cocycles, then any morphism of complexes
$P^\bu\rarrow J^\bu$ is homotopic to zero.
\end{thm}

 Theorem~\ref{fp-proj-fp-inj-objects}(a) is an equivalent restatement
of Theorem~\ref{fp-proj-periodic-object}.
 In the language of~\cite[Definition~3.3]{Gil},
Theorem~\ref{fp-proj-fp-inj-objects}(b) tells
that any complex of fp\+projective objects in a locally finitely
presentable abelian category is dg\+fp\+projective.
 Our proof of Theorem~\ref{fp-proj-fp-inj-objects}(a) deduces it from
Theorem~\ref{fp-proj-fp-inj-objects}(b), while the latter, in turn, is
obtained by reducing the problem to
Theorem~\ref{stovicek-pure-projective-objects}(b) stated below.

 The classes of fp\+projective and weakly fp\+projective objects in
$\sK$ coincide if and only if the category $\sK$ is locally coherent.
 In any locally finitely presentable abelian category that is not
locally coherent, a non-fp-projective (but weakly fp\+projective!)
$\FpProj$\+periodic object exists.
 Here a locally finitely presentable abelian category $\sK$ is said
to be \emph{locally coherent} if the kernel of any morphism acting
between two finitely presentable objects in $\sK$ is finitely
presentable~\cite[Section~2]{Ro}.
 We refer to~\cite[Section~13]{PS3} or~\cite[Sections~8.1\+-8.2]{PS5}
for an additional discussion of locally finitely presentable
and locally coherent abelian categories.

\subsection{{}}
 Let us now formulate the categorical versions of the pure-projective
periodicity results that we use.
 A short exact sequence in an locally finitely presentable abelian
category $\sK$ is called \emph{pure} if it stays exact after applying
the functor $\Hom_\sK(T,{-})$ from any finitely presentable object
$T\in\sK$.
 The class of all pure exact sequences defines an exact category
structure on $\sK$, called the \emph{pure exact structure}.
 Hence the notions of \emph{pure acyclic complexes}, as well as
\emph{pure-projective} and \emph{pure-injective objects}, in a locally
finitely presentable abelian category~$\sK$.
 Given a class of objects $\sA\subset\sK$, an object $M\in\sK$ is
said to be \emph{pure\/ $\sA$\+periodic} if there exists a pure short
exact sequence $0\rarrow M\rarrow A\rarrow M\rarrow0$ with $A\in\sA$.

\begin{thm}[\v St\!'ov\'\i\v cek~{\cite[Corollary~5.5]{Sto}}]
\label{stovicek-pure-proj-periodic-object}
 Let\/ $\sK$ be a locally finitely presentable abelian category;
denote by\/ $\PProj$ the class of all pure-projective objects
in\/~$\sK$.
 Then any pure\/ $\PProj$\+periodic object in\/ $\sK$ is
pure-projective.
\end{thm}

\begin{thm}[\v St\!'ov\'\i\v cek~{\cite[Theorem~5.4 and
Corollary~5.5]{Sto}}]
\label{stovicek-pure-projective-objects}
 Let\/ $\sK$ be a locally finitely presentable abelian category.
Then \par
\textup{(a)} any pure acyclic complex of pure-projective objects in\/
$\sK$ is contractible; \par
\textup{(b)} moreover, if $P^\bu$ is a complex of pure-projective
objects and $X^\bu$ is a pure acyclic complex in\/~$\sK$, then
any morphism of complexes $P^\bu\rarrow X^\bu$ is homotopic to zero.
\end{thm}

 Part~(a) of Theorem~\ref{stovicek-pure-projective-objects} is
obviously a particular case of part~(b).
 Theorem~\ref{stovicek-pure-projective-objects}(a) is an equivalent 
restatement of Theorem~\ref{stovicek-pure-proj-periodic-object}.

 Theorems~\ref{stovicek-pure-proj-periodic-object}
and~\ref{stovicek-pure-projective-objects} are still particular cases
of the exposition in~\cite{Sto}, which is written in the yet more
general setting of finitely accessible additive (not necessarily
abelian) categories.
 Nevertheless, these results are deduced in~\cite{Sto} from
Neeman's theorem stated above as Theorem~\ref{neeman-flat-projective},
which is a result about complexes of modules.

 So let us briefly repeat again the idea of the argument allowing one
to pass from Theorems~\ref{benson-goodearl-flat-projective}
and~\ref{neeman-flat-projective} to
Theorems~\ref{stovicek-pure-proj-periodic-object}
and~\ref{stovicek-pure-projective-objects}.
 For this purpose, one needs to interpret the essentially
small additive category of finitely presentable objects in $\sK$ as
a ``ring with many objects''~$\cR$.
 Then pure-projective objects of $\sK$ are the same things as projective
right $\cR$\+modules, while the whole category $\sK$ is equivalent to
the category of flat right $\cR$\+modules, and pure exact sequences
in $\sK$ correspond to exact sequences of flat right $\cR$\+modules.
 Applying Theorems~\ref{benson-goodearl-flat-projective}
and~\ref{neeman-flat-projective} to right $\cR$\+modules yields
Theorems~\ref{stovicek-pure-proj-periodic-object}
and~\ref{stovicek-pure-projective-objects} for the category~$\sK$.

 We refrain from formulating here the category-theoretic generalizations
of Theorems~\ref{stovicek-fp-injective-inj-periodic-module}\+-%
\ref{stovicek-fp-injective-injective-modules} and
\ref{stovicek-pure-inj-periodic-module}\+-%
\ref{stovicek-pure-injective-modules} (referring the reader instead to
the original exposition in~\cite[Theorem~5.4 and Corollary~5.5]{Sto}),
as these results are not used in our proofs in this paper.
 Let us only point out that (in contrast with the discussion in
the previous paragraph and in the end of
Section~\ref{introd-pure-periodicity}), these results concerning
the fp\+injective and pure-injective periodicity in locally finitely
presentable and locally accessible categories do \emph{not} seem to be
readily deducible from their particular cases for module categories.
 The reason is that for the category of right $R$\+modules there is
the accompanying category of left $R$\+modules, but an arbitrary
locally finitely presentable abelian category does not have such
a (well-behaved) counterpart.

\subsection{{}}
 Section~\ref{cotorsion-pairs-secn} contains preliminary material on
cotorsion pairs in abelian categories, particularly in Grothendieck
categories.
 The preliminaries are continued in
Section~\ref{loc-fin-pres-secn}, where we spell out the definitions
of various classes of objects in locally finitely presentable
abelian categories and their basic properties.
 Section~\ref{hill-lemma-secn} demonstrates the utility of the Hill
lemma for filtrations by finitely presentable objects in locally
finitely presentable abelian categories.

 The main results, including the weak fp\+projectivity of
$\FpProj$\+periodic objects and modules, the failure of fp\+projectivity
of $\FpProj$\+periodic objects in absence of local coherence, etc.,
are formulated and proved in Section~\ref{proofs-main-results-secn}.
 We explain how to use the fp\+projective periodicity to describe
the unbounded derived category of a locally coherent abelian category
in term of complexes of fp\+projective objects in
Section~\ref{derived-categories-secn}.
 Finally, a variety of counterexamples to non-pure
$\PProj$\+periodicity, including counterexamples of $\Proj$\+periodic
objects, are presented in Section~\ref{counterexamples-secn}.

\subsection*{Acknowledgement}
 The authors are grateful to Jan \v Saroch, Jan \v St\!'ov\'\i\v cek,
and James Gillespie for very helpful discussions and communications.
 We also wish to thank the anonymous referee for careful reading of
the manuscript and several helpful suggestions.
 This work is supported by the GA\v CR project 20-13778S
and research plan RVO:~67985840.

\Section{Cotorsion Pairs in Grothendieck Categories}
\label{cotorsion-pairs-secn}

 In this paper we are interested in Grothendieck categories, 
i.~e., cocomplete abelian categories with exact functors of directed
colimits and a set of generators.
 Any locally finitely presentable abelian category is Grothendieck,
while any Grothendieck category is locally presentable.
 All the definitions and results below in this section, suitably
stated, apply to arbitrary locally presentable abelian categories,
as explained in the papers~\cite{PR,PS4}.

 Let $\lambda$~be a regular cardinal and $\sK$ be a cocomplete abelian
category.
 Recall that an object $S\in\sK$ is said to be
\emph{$\lambda$\+presentable} if the functor
$\Hom_\sK(S,{-})\:\sK\rarrow\Sets$ preserves $\lambda$\+directed
colimits, or equivalently, the functor $\Hom_\sK(S,{-})\:\sK\rarrow
\Ab$ preserves $\lambda$\+directed colimits.
 The category $\sK$ is called \emph{locally $\lambda$\+presentable} if
it has a set of generators consisting of $\lambda$\+presentable objects.
 The category $\sK$ is called \emph{locally presentable} if it is
locally $\lambda$\+presentable for some regular cardinal~$\lambda$.
 We refer to the book~\cite{AR} for the background discussion of
presentable objects and locally presentable categories in the general
(nonadditive) category-theoretic context.

 Let $\sK$ be an abelian category.
 Given a class of objects $\sA\subset\sK$, one denotes by
$\sA^{\perp_1}\subset\sK$ the class of all objects $X\in\sK$ such that
$\Ext_\sK^1(A,X)=0$ for all $A\in\sA$.
 Dually, given a class of objects $\sB\subset\sK$, the notation
${}^{\perp_1}\sB\subset\sK$ stands for the class of all objects
$Y\in\sK$ such that $\Ext_\sK^1(Y,B)=0$ for all $B\in\sB$.

 A pair of classes of objects $(\sA,\sB)$ in $\sK$ is called
a \emph{cotorsion pair} if $\sB=\sA^{\perp_1}$ and
$\sA={}^{\perp_1}\sB$.
 Given an arbitrary class of objects $\sS\subset\sK$,
the cotorsion pair $(\sA,\sB)$ in $\sK$ with $\sB=\sS^{\perp_1}$
is said to be \emph{generated} by the class~$\sS$.

 A class of objects $\sA\subset\sK$ is said to be \emph{generating} if
every object of $\sK$ is a quotient object of an object from~$\sA$.
 Dually, a class of objects $\sB\subset\sK$ is said to be 
\emph{cogenerating} if every object of $\sK$ is a subobject of
an object from~$\sB$.

 Let $(\sA,\sB)$ be a cotorsion pair in $\sK$ such that the class $\sA$
is generating and the class $\sB$ is cogenerating.
 Under these assumptions, a cotorsion pair $(\sA,\sB)$ in $\sK$ is said
to be \emph{hereditary} if $\Ext_\sK^2(A,B)=0$ for all objects $A\in\sA$
and $B\in\sB$, or equivalently, $\Ext_\sK^n(A,B)=0$ for all $A\in\sA$,
\,$B\in\sB$, and all integers $n\ge1$.
 Equivalently, a cotorsion pair $(\sA,\sB)$ is hereditary if and only
if the class $\sA$ is closed under the kernels of epimorphisms
in $\sK$, and if and only if the class $\sB$ is closed under
the cokernels of monomorphisms in~$\sK$.
 We refer to~\cite[Lemma~1.4]{PS4} and the references therein for this
well-known characterization of hereditary cotorsion pairs.

 A cotorsion pair $(\sA,\sB)$ in $\sK$ is said to be \emph{complete} if,
for every object $K\in\sK$, there exist short exact sequences
\begin{gather}
 0\lrarrow B'\lrarrow A\lrarrow K \lrarrow0
 \label{special-precover-sequence} \\
 0\lrarrow K \lrarrow B\lrarrow A'\lrarrow0
 \label{special-preenvelope-sequence}
\end{gather}
in $\sK$ with $A$, $A'\in\sA$ and $B$, $B'\in\sB$.

 The short exact sequences~(\ref{special-precover-sequence}\+-%
\ref{special-preenvelope-sequence}) are collectively referred to as
\emph{approximation sequences}.
 The short exact sequence~\eqref{special-precover-sequence} is called
a \emph{special precover sequence}, and the short exact
sequence~\eqref{special-preenvelope-sequence} is called
a \emph{special preenvelope sequence}.

 Let $\sK$ be a Grothendieck category and $\alpha$~be an ordinal.
 An \emph{$\alpha$\+indexed filtration} on an object $F\in\sK$ is
a family of subobjects $(F_i\subset F)_{0\le i\le\alpha}$ such that
\begin{itemize}
\item $F_0=0$ and $F_\alpha=F$;
\item $F_i\subset F_j$ whenever $0\le i\le j\le\alpha$;
\item $F_j=\bigcup_{i<j}F_i$ for every limit ordinal $j\le\alpha$.
\end{itemize}

 An object $F\in\sK$ endowed with an ordinal-indexed filtration
$(F_i\subset F)_{0\le i\le\alpha}$ is said to be \emph{filtered by}
the quotient objects $(F_{i+1}/F_i\in\sK)_{0\le i<\alpha}$.
 In an alternative terminology, the object $F$ is said to be
a \emph{transfinitely iterated extension} of the objects $F_{i+1}/F_i$,
\ $0\le i<\alpha$ (\emph{in the sense of the directed colimit}).
 Given a class of objects $\sS\subset\sK$, one denotes by
$\Fil(\sS)\subset\sK$ the class of all objects of $\sK$ filtered by
(objects isomorphic to) objects from~$\sS$.

 The following result is known as
the \emph{Eklof lemma}~\cite[Lemma~6.2]{GT}.

\begin{lem} \label{eklof-lemma}
 For any class of objects\/ $\sB\subset\sK$, the class\/
${}^{\perp_1}\sB\subset\sK$ is closed under transfinitely iterated
extensions.
 In other words, ${}^{\perp_1}\sB=\Fil({}^{\perp_1}\sB)$.
\end{lem}

\begin{proof}
 This assertion, properly understood, holds in any abelian
category~\cite[Lemma~4.5]{PR}, \cite[Proposition~1.3]{PS4}, and even
in any exact category.
 For an exposition in the generality of Grothendieck categories and
their exact category analogues, see~\cite[Proposition~2.12]{SaoSt}
or~\cite[Proposition~5.7]{Sto-ICRA}.
\end{proof}

 Given a class of objects $\sF\subset\sK$, we denote by
$\sF^\oplus\subset\sK$ the class of all direct summands of the objects
from~$\sF$.
 The next result is called the \emph{Eklof--Trlifaj
theorem}~\cite[Theorem~6.11 and Corollary~6.13]{GT}.

\begin{thm} \label{eklof-trlifaj-theorem}
 Let $\sK$ be a Grothendieck category, $\sS\subset\sK$ be a \emph{set}
of objects, and $(\sA,\sB)$ be the cotorsion pair generated by\/~$\sS$
in\/~$\sK$.
 Then \par
\textup{(a)} if the class\/ $\sA\subset\sK$ is generating, then
the cotorsion pair $(\sA,\sB)$ is complete; \par
\textup{(b)} if the class $\Fil(\sS)\subset\sK$ is generating, then\/
$\sA=\Fil(\sS)^\oplus$.
\end{thm}

\begin{proof}
 A suitable version of this result holds for any locally presentable
abelian category~$\sK$ \,\cite[Corollary~3.6 and Theorem~4.8]{PR},
\cite[Theorems~3.3 and~3.4]{PS4}.
 For an exposition in the generality of Grothendieck categories and
their exact category generalizations, see~\cite[Corollary~2.15]{SaoSt}
or~\cite[Theorem~5.16]{Sto-ICRA}.

 The difference is that, for Grothendieck categories~$\sK$, one does not
need to assume the class $\sB$ to be cogenerating, as this holds
automatically for any cotorsion pair $(\sA,\sB)$ (because there are
enough injective objects in $\sK$ and all of them belong to~$\sB$).
 Similarly, the assumption that the class $\sA$ is generating becomes
redundant when there are enough projective objects in~$\sK$.
\end{proof}

 Given a class of objects $\sA\subset\sK$, denote by
$\sA^{\perp_{\ge1}}\subset\sK$ the class of all objects $X\in\sK$
such that $\Ext^n_\sK(A,X)=0$ for all $A\in\sA$ and $n\ge1$.
 Dually, given a class of objects $\sB\subset\sK$, let
${}^{\perp_{\ge1}}\sB\subset\sK$ denote the class of all objects
$Y\in\sK$ such that $\Ext^n_\sK(Y,B)=0$ for all $B\in\sB$ and $n\ge1$.

 A class of objects $\sS\subset\sK$ is said to be \emph{self-generating}
if for any epimorphism $K\rarrow S$ in $\sK$ with $S\in\sS$ there
exist an epimorphism $S'\rarrow S$ in $\sK$ with $S'\in\sS$ and
a morphism $S'\rarrow K$ making the triangular diagram $S'\rarrow K
\rarrow S$ commutative.
 Clearly, any generating class of objects is self-generating.

\begin{lem} \label{generated-by-self-generating}
 Let\/ $\sK$ be an abelian category and\/ $\sS\subset\sK$ be
a self-generating class of objects closed under the kernels of
epimorphisms.
 Then\/ $\sS^{\perp_1}=\sS^{\perp_{\ge1}}\subset\sK$.
\end{lem}

\begin{proof}
 This is a partial generalization of the standard characterization of
hereditary cotorsion pairs in abelian categories~\cite[Lemma~1.4]{PS4}
mentioned above.
 The argument from~\cite[Lemma~6.17]{Sto-ICRA}
or~\cite[Lemma~4.25]{SaoSt} applies.
 To give some details, it follows from the assumptions of the lemma
that for any objects $S\in\sS$ and $K\in\sK$, and any Yoneda extension
class $\xi\in\Ext_\sK^n(S,K)$ with $n\ge2$ there exists a short exact
sequence $0\rarrow S''\rarrow S'\rarrow S\rarrow0$ in $\sK$ with
$S'$, $S''\in\sS$ such that the class~$\xi$ comes from a certain
Yoneda extension class $\eta\in\Ext_\sK^{n-1}(S'',K)$.
 Thus $\Ext_\sK^{n-1}(S'',K)=0$ for all $S''\in\sS$ implies
$\Ext_\sK^n(S,K)=0$.
\end{proof}

\begin{lem} \label{hereditary-generated-lemma}
 Let\/ $\sK$ be an abelian category with enough injective objects,
and let\/ $\sT\subset\sK$ be a class of objects.
 Put\/ $\sB=\sT^{\perp_{\ge1}}$.
 Then \par
\textup{(a)} ${}^{\perp_1}\sB={}^{\perp_{\ge1}}\sB\subset\sK$; \par
\textup{(b)} if the class\/ $\sA={}^{\perp_1}\sB={}^{\perp_{\ge1}}\sB$
is generating in\/ $\sK$, then $(\sA,\sB)$ is a hereditary
cotorsion pair in\/~$\sK$.
\end{lem}

\begin{proof}
 Part~(a) follows from the observation that the class\/ $\sB$ is
\emph{coresolving} in~$\sK$.
 In other words, the class $\sB\subset\sK$ contains the injective
objects and is closed under extensions and the cokernels
of monomorphisms.
 To prove part~(b), one observes that the class $\sA$ is
\emph{resolving} in $\sK$, i.~e., it is generating and
closed under extensions and the kernels of epimorphisms.
 Therefore, $\sA^{\perp_{\ge1}}=\sA^{\perp_1}$ by
Lemma~\ref{generated-by-self-generating}.
 It follows immediately from the constructions that
$\sB=\sA^{\perp_{\ge1}}$, so we are done.
\end{proof}

\begin{prop} \label{hereditary-generated-prop}
 Let\/ $\sK$ be a Grothendieck category and\/ $\sT\subset\sK$ be a set
of objects.
 Put\/ $\sB=\sT^{\perp_{\ge1}}$ and\/
$\sA={}^{\perp_1}\sB={}^{\perp_{\ge1}}\sB$, as per
Lemma~\ref{hereditary-generated-lemma}, and assume that
the class\/ $\sA$ is generating in\/~$\sK$.
 Then $(\sA,\sB)$ is a hereditary complete cotorsion pair in\/ $\sK$
generated by a certain set of objects\/~$\sS$.
\end{prop}

\begin{proof}
 In view of Lemma~\ref{hereditary-generated-lemma} and
Theorem~\ref{eklof-trlifaj-theorem}, we only need to construct a set
of objects $\sS\subset\sK$ such that $\sS^{\perp_1}=\sT^{\perp_{\ge1}}$.
 Clearly, we have $\sT\subset\sA$.
 Arguing by induction, it suffices to show that for every object
$S\in\sA$ and an integer $n\ge2$ there exists a set of objects
$\sS'\subset\sA$ such that for any given $X\in\sK$ one has
$\Ext_\sK^n(S,X)=0$ whenever $\Ext_\sK^{n-1}(S',X)=0$
for all $S'\in\sS'$.

 Let $\lambda$~be a regular cardinal such that the category $\sK$ is
locally $\lambda$\+presentable and the object $S$ is
$\lambda$\+presentable.
 For every $\lambda$\+presentable object $K\in\sK$ endowed with
an epimorphism $K\rarrow S$, choose an epimorphism $A\rarrow K$ onto
$K$ from an object $A\in\sA$, and set $S'$ to be the kernel of
the composition $A\rarrow K\rarrow S$.
 Then one has $S'\in\sA$, since the class $\sA$ is closed under
the kernels of epimorphisms.

 Let $\sS'$ be the set of all objects $S'$ obtained in this way.
 For any Yoneda extension class $\xi\in\Ext_\sK^n(S,X)$, there exists
a short exact sequence $0\rarrow Z\rarrow Y\rarrow S\rarrow0$ in $\sK$
such that the class~$\xi$ comes from a Yoneda extension class
$\eta\in\Ext_\sK^{n-1}(Z,X)$.
 By~\cite[Lemma~3.4]{PR}, any short exact sequence $0\rarrow Z\rarrow
Y\rarrow S\rarrow0$ in $\sK$ is a pushout of a short exact sequence
$0\rarrow Z'\rarrow K\rarrow S\rarrow0$ in $\sK$ in which the object $K$
is $\lambda$\+presentable.
 The latter short exact sequence is, in turn, a pushout of a short
exact sequence $0\rarrow S'\rarrow A\rarrow S\rarrow0$ with $A\in\sA$
and $S'\in\sS'$.
 It follows easily that $\Ext_\sK^{n-1}(S',X)$ for all $S'\in\sS'$
implies $\Ext_\sK^n(S,X)=0$.
\end{proof}

 For any additive/abelian category $\sK$, let us denote by $\sC(\sK)$
the additive/abelian category of complexes in~$\sK$ and by $\Hot(\sK)$
the triangulated homotopy category of (complexes in)~$\sK$.
 When a category $\sK$ is locally $\lambda$\+presentable or
Grothendieck, so is the category~$\sC(\sK)$.
 As usual, we denote by $C^\bu\longmapsto C^\bu[n]$ the cohomological
grading shifts of a complex $C^\bu$; so $C^\bu[n]^i=C^{i+n}$ for all
$n$, $i\in\boZ$.
 The following lemma is well-known and very useful.

\begin{lem} \label{complexes-homotopy-hom-ext-lemma}
 For any two complexes $A^\bu$ and $B^\bu$ in an abelian
category\/~$\sK$, the group\/ $\Hom_{\Hot(\sK)}(A^\bu,B^\bu[1])$ is
naturally isomorphic to the subgroup in\/
$\Ext^1_{\sC(\sK)}(A^\bu,B^\bu)$ formed by all the degree-wise split
extension classes.
 In particular, if\/ $\Ext_\sK^1(A^i,B^i)=0$ for all $i\in\boZ$, then
$$
 \Hom_{\Hot(\sK)}(A^\bu,B^\bu[1])\simeq\Ext^1_{\sC(\sK)}(A^\bu,B^\bu).
$$
\end{lem}

\begin{proof}
 The point is that degree-wise split extensions $0\rarrow B^\bu
\rarrow C^\bu\rarrow A^\bu\rarrow0$ in $\sC(\sK)$ are described as
the cones of morphisms of complexes $f\:A^\bu\rarrow B^\bu[1]$;
specifically, $C=\cone(f)[-1]$.
 Such extensions corresponding to two morphisms $f'$, $f''\:A^\bu
\rarrow B^\bu$ represent the same extension class in
$\Ext^1_{\sC(\sK)}(A^\bu,B^\bu)$ if and only if the two morphisms
$f'$ and~$f''$ are cochain homotopic.
 (Cf.~\cite[Lemma~5.1]{PS4}.)
\end{proof}

\Section{Locally Finitely Presentable Abelian Categories}
\label{loc-fin-pres-secn}

 The definitions of a $\lambda$\+presentable object and
a locally $\lambda$\+presentable (abelian) category were already
given in the beginning of Section~\ref{cotorsion-pairs-secn}.
 The concepts of a \emph{finitely presentable object} and
a \emph{locally finitely presentable category} are obtained by 
specializing to the case of the countable cardinal $\lambda=\aleph_0$.
 Given a locally finitely presentable abelian category $\sK$, we denote
by $\sK_\fp\subset\sK$ the full subcategory of finitely presentable
objects.
 The category $\sK_\fp$ is essentially small.
 The full subcategory $\sK_\fp$ is closed under extensions and cokernels
in~$\sK$.

 More generally, an object $S\in\sK$ is said to be
\emph{finitely generated} if the functor $\Hom_\sK(S,{-})$
preserves the directed colimits of diagrams of monomorphisms in~$\sK$.
 In a locally finitely presentable abelian category, the finitely
generated objects are precisely the quotient objects of the finitely
presentable ones.
 Given a short exact sequence $0\rarrow Q\rarrow S\rarrow T\rarrow0$
in a locally finitely presentable abelian category $\sK$ with
a finitely presentable object $S$, the object $T$ is finitely
presentable if and only if the object $Q$ is finitely generated.

 \emph{Finitely accessible additive categories} (in the terminology
of~\cite{AR}) form a wider class of categories than the locally
finitely presentable abelian ones.
 This class of additive categories, which is natural for many purposes,
was studied in the papers~\cite{CB,Kra} under the name of ``locally
finitely presented additive categories''.

 A locally finitely presentable abelian category $\sK$ is said to be
\emph{locally coherent} if the class of all finitely presentable objects
in $\sK$ is closed under the kernels of epimorphisms, or equivalently,
if it is closed under the kernels of all morphisms.
 A locally finitely presentable abelian category $\sK$ is locally
coherent if and only if any finitely generated subobject of
a finitely presentable object of $\sK$ is finitely presentable, or
equivalently, the kernel of any morphism from a finitely generated
object to a finitely presentable one is
finitely generated~\cite[Section~2]{Ro}.
 We refer to~\cite[Section~13]{PS3} or~\cite[Sections~8.1\+-8.2]{PS5}
for a further discussion of locally finitely presentable
and locally coherent abelian categories.

 For example, for any associative ring $R$, the category of $R$\+modules
$\Modr R$ is locally finitely presentable.
 More generally, for any small preadditive category $\cR$ (i.~e.,
a small category enriched in abelian groups), one denotes by
$\Modr\cR=\Funct_\ad(\cR^\sop,\Ab)$ the category of contravariant
additive functors from $\cR$ to the category of abelian groups $\Ab$,
and by $\cR\Modl=\Funct_\ad(\cR,\Ab)$ the category of covariant
additive functors $\cR\rarrow\Ab$.
 Both $\Modr\cR$ and $\cR\Modl$ are locally finitely presentable
abelian categories, with the full subcategories of finitely
presentable objects consisting of all the cokernels of arbitrary
morphisms between finite direct sums of (co)representable functors.

 The direct summands of coproducts of (co)representable functors are
the projective objects in $\Modr\cR$ and $\cR\Modl$.
 Generally, $\cR$ can be viewed as a ``ring with many objects'' or
``a nonunital ring with enough idempotents''; then the objects of
$\Modr\cR$ and $\cR\Modl$ are simply interpreted as right and left
$\cR$\+modules.
 Essentially all the constructions and results of the conventional
module theory can be easily transferred to modules over rings with
many objects.
 In particular, there is a naturally defined tensor product functor
$\ot_\cR\:\Modr\cR\times\cR\Modl\rarrow\Ab$, and its derived functor
$\Tor^\cR_*$ can be constructed as usual.
 Hence one can speak of \emph{flat} right and left $\cR$\+modules;
these are precisely the directed colimits of projective ones.
 We denote the full subcategory of flat modules by
$\Modrfl\cR\subset\Modr\cR$.

 Let us recall some definitions sketched in the introduction.
 In a locally finitely presentable abelian category $\sK$, an object
$J$ is said to be \emph{fp\+injective} if $\Ext^1_\sK(T,J)=0$ for
all finitely presentable objects $T\in\sK$.
 An object $P\in\sK$ is said to be \emph{fp\+projective} if
$\Ext^1_\sK(P,J)=0$ for all fp\+injective objects $J\in\sK$.

 We denote the full subcategory of fp\+injective objects by
$\sK_\inj^\fp\subset\sK$ and the full subcategory of
fp\+projective objects by $\sK_\proj^\fp\subset\sK$.
 So $(\sK_\proj^\fp$, $\sK_\inj^\fp)$ is the cotorsion pair
generated by $\sS=\sK_\fp$ in~$\sK$.
 Applying Theorem~\ref{eklof-trlifaj-theorem}, one easily concludes
that this cotorsion pair is complete and $\sK_\proj^\fp=
\Fil(\sK_\fp)^\oplus$ (since any object of $\sK$ is a quotient
object of a coproduct of finitely presentables).

 An object $J\in\sK$ is said to be \emph{strongly fp\+injective} if
$\Ext^n_\sK(T,J)=0$ for all $T\in\sK_\fp$ and $n\ge1$.
 An object $P\in\sK$ is said to be \emph{weakly fp\+projective} if
$\Ext^1_\sK(P,J)=0$ for all strongly fp\+injective objects $J\in\sK$,
or equivalently, $\Ext^n_\sK(P,J)=0$ for all strongly fp\+injective
objects $J\in\sK$ and all $n\ge1$ (cf.\
Lemma~\ref{hereditary-generated-lemma}).

 We denote the full subcategory of strongly fp\+injective objects by
$\sK_\inj^\sfp\subset\sK$ and the full subcategory of weakly
fp\+projective objects by $\sK_\proj^\wfp\subset\sK$.
 By Proposition~\ref{hereditary-generated-prop} applied to
(a~representative set of isomorphism classes of objects in)
$\sT=\sK_\fp$, the pair of classes of objects 
$(\sK_\proj^\wfp$, $\sK_\inj^\sfp)$ is a hereditary complete
cotorsion pair in~$\sK$.

 A short exact sequence $0\rarrow K\rarrow L\rarrow M\rarrow0$ in
a locally finitely presentable abelian category $\sK$ is said to be
\emph{pure} if the induced morphism of abelian groups
$\Hom_\sK(T,L)\rarrow \Hom_\sK(T,M)$ is surjective for all
finitely presentable objects $T\in\sK$.
 In this case, the morphism $K\rarrow L$ is called a \emph{pure
monomorphism} and the morphism $L\rarrow M$ is called a \emph{pure
epimorphism} in~$\sK$.
 The object $K$ is said to be a \emph{pure subobject} of $L$,
and the object $M$ is said to be a \emph{pure quotient} of~$L$.
 Acyclic complexes obtained by splicing pure short exact sequences
are called \emph{pure acyclic} or \emph{pure exact}.
 The class of all pure short exact sequences defines an exact
category structure on $\sK$, called the \emph{pure exact structure}.

 An object $P\in\sK$ is said to be \emph{pure-projective} if it is
projective with respect to the pure exact structure, i.~e.,
the induced map $\Hom_\sK(P,L)\rarrow\Hom_\sK(P,M)$ is surjective
for any pure epimorphism $L\rarrow M$ in~$\sK$.
 There are enough pure-projective objects in~$\sK$: any object is
a pure quotient of a pure-projective one.
 An object $P\in\sK$ is pure-projective if and only if it is
a direct summand of a coproduct of finitely presentable objects.
 So any pure-projective object is fp\+projective, but the converse
usually does \emph{not} hold.

 We denote the class of all pure-projective objects by
$\sK_\proj^\pur\subset\sK$.
 The \emph{pure-injective objects} are defined dually, but we will not
use them in this paper.

 The next two lemmas are well-known.
 The following one explains why fp\+injective objects are often called
``absolutely pure''.

\begin{lem} \label{absolutely-pure}
 Let\/ $\sK$ be a locally finitely presentable abelian category.
 Then an object $J\in\sK$ is fp\+injective if and only if any
monomorphism $J\rarrow K$ from $J$ to any object $K\in\sK$ is pure.
\end{lem}

\begin{proof}
 ``If'': let $0\rarrow J\rarrow K\rarrow T\rarrow0$ be a short exact
sequence in $\sK$ with $T\in\sK_\fp$.
 By assumption, this short exact sequence is pure.
 Hence the map $\Hom_\sK(T,K)\rarrow\Hom_\sK(T,T)$ is surjective, so
our short exact sequence splits.
 We have shown that $\Ext^1_\sK(T,J)=0$, as desired.

 ``Only if'': let $0\rarrow J\rarrow K\rarrow M\rarrow0$ be a short
exact sequence in~$\sK$.
 Then the assumption of $\Ext^1_\sK(T,J)=0$ for any finitely
presentable object $T\in\sK$ implies the desired surjectivity of
the map $\Hom_\sK(T,K)\rarrow\Hom_\sK(T,M)$.
\end{proof}

\begin{lem} \label{pure-exact-structure-as-flat-modules}
 Let\/ $\sK$ be a locally finitely presentable abelian category.
 Denote by\/ $\cR$ a small category equivalent to\/~$\sK_\fp$.
 Then the functor assigning to an object $K\in\sK$ the contravariant
functor\/ $\Hom_\sK({-},K)\:\sK^\sop\rarrow\Ab$ restricted to
the full subcategory\/ $\sK_\fp\subset\sK$ defines an equivalence
between\/ $\sK$ and the full subcategory of \emph{flat} modules
in the category of right $\cR$\+modules\/~$\Modr\cR$,
$$
 \sK\simeq\Modrfl\cR.
$$
 Under this equivalence, the pure exact structure on\/ $\sK$
corresponds to the exact structure on\/ $\Modrfl\cR$ inherited from
the abelian exact structure on\/ $\Modr\cR$.
 The pure-projective objects of\/ $\sK$ correspond to the projective
objects of\/~$\Modr\cR$.
\end{lem}

\begin{proof}
 This is a standard observation.
 Notice that the functor $K\longmapsto\Hom_\sK({-},K)|_{\sK_\fp}$
identifies the full subcategory of finitely presentable objects
$\sK_\fp\subset\sK$ with the full subcategory of representable
functors in $\Funct_\ad(\sK_\fp^\sop,\Ab)=\Modr\cR$.
 For an arbitrary small preadditive category $\cR$, the representable
functors play the role of free modules with one generator in $\Modr\cR$.
 When $\cR$ is an idempotent-complete additive category, as in
the situation at hand, these are the same things as the finitely
generated projective modules.
 It remains to recall that the objects of $\sK$ are the directed
colimits of the objects from $\sK_\fp$, while the flat modules are
the directed colimits of the finitely generated projective modules, etc.
\end{proof}

\Section{Two Instances of the Hill Lemma}
\label{hill-lemma-secn}

 The \emph{Hill lemma}~\cite[Theorem~7.10]{GT}, \cite[Theorem~2.1]{Sto0}
is a general property of modules or Grothendieck category objects with
ordinal-indexed filtrations, which becomes particularly important when
the indexing ordinal is large as compared to the presentability ranks
of the successive quotient modules/objects in the filtration.
 The Hill lemma tells that, given one such filtration on a particular
object, one can produce many similar filtrations.

 In this paper we apply the Hill lemma in locally finitely presentable
abelian categories~$\sK$.
 We do not reproduce the lengthy general formulation of the Hill lemma
(referring the reader instead to~\cite{GT,Sto0}), but only state two
particular cases or corollaries that are relevant for our purposes.

\begin{lem} \label{fin-gen-subobject-hill-lemma}
 Let\/ $\sK$ be a locally finitely presentable abelian category and\/
$\sS\subset\sK_\fp$ be a class of finitely presentable objects closed
under extensions.
 Let $P\in\Fil(\sS)\subset\sK$ be an\/ $\sS$\+filtered object
(as defined in Section~\ref{cotorsion-pairs-secn}), and let
$Q\subset P$ be a finitely generated subobject.
 Then there exists an intermediate subobject $Q\subset S\subset P$
such that $S\in\sS$.
\end{lem}

\begin{proof}
 This is a particular case of~\cite[Theorem~2.1\,(H3\+-H4)]{Sto0}
applied in the case of the countable cardinal $\kappa=\aleph_0$.
\end{proof}

\begin{cor} \label{fp-projective-fin-generated-cor}
\textup{(a)} In a locally finitely presentable abelian category,
any finitely generated fp\+projective object is finitely
presentable. \par
\textup{(b)} In a locally coherent abelian category, any finitely
generated subobject of an fp\+projective object is finitely
presentable.
\end{cor}

\begin{proof}
 To prove part~(a), let $Q$ be a finitely generated fp\+projective
object in a locally finitely presentable abelian category~$\sK$.
 Then $Q$ is a direct summand of a $\sK_\fp$\+filtered object $P$,
so we have two morphisms $Q\rarrow P\rarrow Q$ with the composition
equal to~$\id_Q$.
 By Lemma~\ref{fin-gen-subobject-hill-lemma}, there exists a finitely
presentable subobject $S\subset P$ such that the morphism $Q\rarrow P$
factorizes as $Q\rarrow S\rarrow P$.
 It follows that $Q$ is a direct summand of~$S$, hence $Q$ is also
finitely presentable.

 Part~(b): Let $Q$ be a finitely generated subobject of
an fp\+projective object in a locally coherent category~$\sK$.
 Then $Q$ is also a subobject of a $\sK_\fp$\+filtered object.
 By Lemma~\ref{fin-gen-subobject-hill-lemma}, it follows that $Q$
is a subobject of a finitely presentable object $S\in\sK_\fp$.
 It remains to recall that in a locally coherent category any finitely
generated subobject of a finitely presentable object is finitely
presentable.

 Both the assertions~(a) and~(b) are also provable without
the Hill lemma.
 For a proof of~(a) (in the case of module categories),
see~\cite[Theorem~2.1.10]{Gla}.
 A proof of~(b) can be found in~\cite[Lemma~1.5]{Pfp}.
 The arguments above are particularly neat and transparent, though.
\end{proof}

\begin{cor} \label{two-fp-cotors-pairs-agree-equiv-loc-coherent}
 The following conditions are equivalent for a locally finitely
presentable abelian category\/~$\sK$:
\begin{enumerate}
\item the cotorsion pair $(\sK_\proj^\fp$, $\sK_\inj^\fp)$ is
hereditary in\/~$\sK$;
\item all weakly fp\+projective objects in\/ $\sK$ are fp\+projective;
\item all fp\+injective objects in\/ $\sK$ are strongly fp\+injective;
\item the category\/ $\sK$ is locally coherent.
\end{enumerate}
\end{cor}

\begin{proof}
 (1)\,$\Longrightarrow$\,(4)
 It suffices to show that the kernel $Q$ of any epimorphism $S\rarrow T$
between finitely presentable objects $S$, $T\in\sK$ is finitely
presentable.
 Notice that the object $Q$ is always finitely generated.
 The objects $S$ and $T$ are fp\+projective.
 Since the cotorsion pair $(\sK_\proj^\fp,\sK_\inj^\fp)$ is hereditary
by assumption, it follows that the object $Q$ is fp\+projective.
 It remains to invoke Corollary~\ref{fp-projective-fin-generated-cor}(a)
in order to conclude that $Q$ is finitely presentable.

 (4)\,$\Longrightarrow$\,(3)
 The class $\sS=\sK_\fp$ of all finitely presentable objects is
self-generating in any locally finitely presentable abelian
category~$\sK$.
 In a locally coherent category, it is also closed under the kernels
of epimorphisms, so Lemma~\ref{generated-by-self-generating} applies.

 (1)\,$\Longrightarrow$\,(3)\,$\Longrightarrow$\,(2) hold by
the definitions.
 
 (2)\,$\Longrightarrow$\,(3) holds because both
$(\sK_\proj^\fp$, $\sK_\inj^\fp)$ and
$(\sK_\proj^\wfp$, $\sK_\inj^\sfp)$ are cotorsion pairs in $\sK$,
and a cotorsion pair is determined by its left class.
 
 (2) $+$~(3) $\Longrightarrow$~(1) holds because the cotorsion pair
$(\sK_\proj^\wfp$, $\sK_\inj^\sfp)$ is always hereditary (as explained
in Section~\ref{loc-fin-pres-secn}).
\end{proof}

 The following lemma plays a crucial role in the proofs of the main
theorems in Section~\ref{proofs-main-results-secn} (specifically,
Theorem~\ref{lfp-category-hot-hom-fp-proj-fp-inj-thm}).

\begin{lem} \label{complexes-deconstructed}
 Let $\sK$ be a locally finitely presentable abelian category and\/
$\sS\subset\sK_\fp$ be a class of finitely presentable objects closed
under extensions.
 Let $P^\bu\in\sC(\Fil(\sS))$ be a complex in\/ $\sK$ whose terms are\/
$\sS$\+filtered objects.
 Then the complex $P^\bu$, viewed as an object of the abelian category
of complexes\/ $\sC(\sK)$, is filtered by bounded below complexes
whose terms belong to\/~$\sS$.
\end{lem}

\begin{proof}
 This is~\cite[(proof of) Proposition~4.3]{Sto0}
for $\kappa=\aleph_0$.
\end{proof}

\Section{Proofs of Main Results}  \label{proofs-main-results-secn}

 The following theorem is the main result of this paper.

\begin{thm} \label{lfp-category-cycles-in-fp-projective-thm}
 Let\/ $\sK$ be a locally finitely presentable abelian category, and
let $P^\bu$ be an acyclic complex in\/ $\sK$ whose terms $P^n$ are
fp\+projective objects for all $n\in\boZ$.
 Denote by $Z^n\in\sK$ the objects of cocycles of the acyclic
complex~$P^\bu$.
 Then all the objects $Z^n$ are weakly fp\+projective, that is
$Z^n\in\sK_\proj^\wfp$ for all $n\in\boZ$.
\end{thm}

 The proof of
Theorem~\ref{lfp-category-cycles-in-fp-projective-thm} is based on
the next Theorem~\ref{lfp-category-hot-hom-fp-proj-fp-inj-thm}.

\begin{thm} \label{lfp-category-hot-hom-fp-proj-fp-inj-thm}
 Let\/ $\sK$ be a locally finitely presentable abelian category,
let $P^\bu\in\sC(\sK_\proj^\fp)$ be a complex of fp\+projective
objects in\/ $\sK$, and let $J^\bu$ be an acyclic complex of
fp\+injective objects in\/ $\sK$ with fp\+injective objects of cocycles.
 Then any morphism of complexes $P^\bu\rarrow J^\bu$ is homotopic
to zero.
\end{thm}

 The proof of
Theorem~\ref{lfp-category-hot-hom-fp-proj-fp-inj-thm}, in turn, is based
on the following Theorem~\ref{stovicek-pure-neeman}.

\begin{thm}[St\!'ov\'\i\v cek~{\cite[Theorem~5.4]{Sto}}]
\label{stovicek-pure-neeman}
 Let\/ $\sK$ be a locally finitely presentable abelian category,
let $P^\bu\in\sC(\sK_\proj^\pur)$ be a complex of pure-projective
objects in\/ $\sK$, and let $X^\bu$ be a pure acyclic complex
in\/~$\sK$.
 Then any morphism of complexes $P^\bu\rarrow X^\bu$ is homotopic
to zero.
\end{thm}

 Finally, the proof of Theorem~\ref{stovicek-pure-neeman} is based on
the next Theorem~\ref{neeman-theorem}.

\begin{thm}[Neeman~{\cite[Theorem~8.6]{Neem}}] \label{neeman-theorem}
 Let\/ $\cR$ be a small preadditive category, let $P^\bu\in
\sC(\Modrproj\cR)$ be a complex of projective objects in\/ $\Modr\cR$,
and let $X^\bu$ be an acyclic complex of flat right\/ $\cR$\+modules
with flat modules of cocycles.
 Then any morphism of complexes $P^\bu\rarrow X^\bu$ is homotopic
to zero.
\end{thm}

\begin{proof}[Proof of Theorem~\ref{neeman-theorem}]
 This is a straightforward generalization
of~\cite[Theorem~8.6\,(iii)\,$\Rightarrow$\,(i)]{Neem} from modules
over the conventional unital rings to modules over ``rings with enough
idempotents'' or (which is essentially the same) ``rings with many
objects'' or (which is the same) small preadditive categories.
 As usual for such generalizations, it is provable by the same method.
\hbadness=2650
\end{proof}

\begin{proof}[Proof of Theorem~\ref{stovicek-pure-neeman}]
 This assertion, stated in the introduction as
Theorem~\ref{stovicek-pure-projective-objects}(b), is one specific
aspect of a particular case of \v St\!'ov\'\i\v cek's
\cite[Theorem~5.4]{Sto}, provable by reduction to Neeman's theorem.
 Let $\cR$ be a small category equivalent to\/~$\sK_\fp$.
 Applying Lemma~\ref{pure-exact-structure-as-flat-modules},
one reduces Theorem~\ref{stovicek-pure-neeman} to
Theorem~\ref{neeman-theorem}.
\end{proof}

\begin{proof}[Proof of
Theorem~\ref{lfp-category-hot-hom-fp-proj-fp-inj-thm}]
 As explained in Section~\ref{loc-fin-pres-secn}, any fp\+projective
object is a direct summand of an object filtered by finitely
presentable ones.
 In the context of the theorem, $P^\bu$ is an arbitrary complex of
fp\+projective objects.
 So a complex $Q^\bu$ in the category $\sK$ can be found such that
the direct sum $P^\bu\oplus Q^\bu$ is a complex whose every term
$P^i\oplus Q^i$, $\,i\in\boZ$, is filtered by finitely presentable
objects (e.~g., one can choose $Q^\bu$ to be a suitable complex
with zero differential).
 If we manage to prove that every morphism of complexes $P^\bu\oplus
Q^\bu\rarrow J^\bu$ is homotopic to zero, it will follow that
every morphism of complexes $P^\bu\rarrow J^\bu$ is homotopic
to zero as well.
 Redenoting $P^\bu\oplus Q^\bu$ by $P^\bu$, we have shown that it
can be assumed, without loss of generality, that every object $P^i$,
\,$i\in\boZ$, is filtered by finitely presentable objects.

 Thus we now suppose that $P^i\in\Fil(\sK_\fp)$.
 By the definition of fp\+projective objects, we have
$\Ext^1_\sK(P^i,J^j)=0$ for all $i$, $j\in\boZ$.
 Therefore, Lemma~\ref{complexes-homotopy-hom-ext-lemma} provides
an isomorphism of abelian groups
$$
 \Hom_{\Hot(\sK)}(P^\bu,J^\bu)\simeq\Ext^1_{\sC(\sK)}(P^\bu,J^\bu[-1]).
$$

 By Lemma~\ref{complexes-deconstructed} (for $\sS=\sK_\fp$),
the complex $P^\bu$ is filtered by (bounded below) complexes of
finitely presentable objects.
 In view of Lemma~\ref{eklof-lemma}, it suffices to show that
$\Ext^1_{\sC(\sK)}(S^\bu,J^\bu[-1])=0$ for any complex
$S^\bu\in\sC(\sK_\fp)$.
 The same isomorphism from
Lemma~\ref{complexes-homotopy-hom-ext-lemma} tells that
$$
 \Ext^1_{\sC(\sK)}(S^\bu,J^\bu[-1])\simeq
 \Hom_{\Hot(\sK)}(S^\bu,J^\bu).
$$
 Essentially, we have reduced the assertion of the proposition to
the case of a complex of finitely presentable objects in place of
a complex of fp\+projective objects.

 Now we observe that, by the definition, all finitely presentable
objects are pure-projective.
 On the other hand, by Lemma~\ref{absolutely-pure}, any acyclic complex
in $\sK$ with fp\+injective objects of cocycles is pure acyclic.
 It remains to apply Theorem~\ref{stovicek-pure-neeman} to $P^\bu=S^\bu$
and $X^\bu=J^\bu$.
\end{proof}

\begin{proof}[Proof of
Theorem~\ref{lfp-category-cycles-in-fp-projective-thm}]
 Let $Y\in\sK_\inj^\sfp$ be a strongly fp\+injective object.
 Given an integer $n\in\boZ$, we have to show that
$\Ext^1_\sK(Z^n,Y)=0$.
 In the short exact sequence
$$
 0\lrarrow Z^{n-1}\lrarrow P^{n-1}\lrarrow Z^n\lrarrow0
$$
we have $\Ext^1_\sK(P^{n-1},Y)=0$; so it suffices to prove that the map
$\Hom_\sK(P^{n-1},Y)\rarrow\Hom_\sK(Z^{n-1},Y)$ is surjective.
 For this purpose, we will show that the complex of abelian groups
$\Hom_\sK(P^\bu,Y)$ is acyclic.

 Let $0\rarrow Y\rarrow I^0\rarrow I^1\rarrow I^2\rarrow\dotsb$ be
an injective resolution of the object $Y$ in the category~$\sK$.
 Denote by $J^\bu$ the acyclic complex $Y\rarrow I^\bu$.
 Since the object $Y$ is strongly fp\+injective and the cotorsion pair
$(\sK_\proj^\wfp,\sK_\inj^\sfp)$ is hereditary in $\sK$, all
the objects of cocycles of the complex $J^\bu$ are (strongly)
fp\+injective.
 All the terms of the complex $J^\bu$ are also obviously (strongly)
fp\+injective.

 Recall that $P^\bu$ is a complex of fp\+projective objects.
 By Theorem~\ref{lfp-category-hot-hom-fp-proj-fp-inj-thm}, it follows
that all morphisms of complexes $P^\bu\rarrow J^\bu[n]$, \,$n\in\boZ$,
are homotopic to zero.
 In other words, this means that the complex of abelian groups
$\Hom_\sK(P^\bu,J^\bu)$ is acyclic.

 On the other hand, for any acyclic complex $X^\bu$ and any bounded
below complex of injective objects $I^\bu$ in an abelian category $\sK$,
the complex of abelian groups $\Hom_\sK(X^\bu,I^\bu)$ is well-known
to be acyclic.
 In the situation at hand, we observe that the complex
$\Hom_\sK(P^\bu,I^\bu)$ is acyclic.

 Since both the complexes of abelian groups $\Hom_\sK(P^\bu,I^\bu)$ and
$\Hom_\sK(P^\bu,J^\bu)$ are acyclic, and the complex $J^\bu$ is simply
the augmented coresolution $J^\bu=(Y\to I^\bu)$, we can finally conclude
that the complex $\Hom_\sK(P^\bu,Y)$ is acyclic, as desired.
\end{proof}

\begin{cor} \label{lfp-category-weakly-fp-periodicity-cor}
 Let\/ $\sK$ be a locally finitely presentable abelian category
and\/ $\FpProj=\sK_\proj^\fp$ be the class of all fp\+projective
objects in\/~$\sK$.
 Then any\/ $\FpProj$\+periodic object in\/ $\sK$ is weakly
fp\+projective.
\end{cor}

\begin{proof}
 We refer to Section~\ref{introd-periodic-defined-subsecn} of
the introduction for the definition of an $\sA$\+periodic object.
 Given a short exact sequence $0\rarrow M\overset k\rarrow A
\overset q\rarrow M\rarrow0$ in $\sK$ with $A\in\sK_\proj^\fp$,
all one needs to do is to splice up a doubly unbounded sequence
of copies of the given short exact sequence, obtaining
an acyclic complex
$$
 \dotsb\lrarrow A\overset{kq}\lrarrow A
 \overset{kq}\lrarrow A\lrarrow\dotsb
$$
and apply Theorem~\ref{lfp-category-cycles-in-fp-projective-thm}
to the resulting complex.
\end{proof}

\begin{cor} \label{lfp-category-cocycles-fp-proj-iff-loc-coherent-cor}
 Let\/ $\sK$ be a locally finitely presentable abelian category.
 Then the following conditions are equivalent:
\begin{enumerate}
\item in any acyclic complex in\/ $\sK$ with fp\+projective terms,
the objects of cocycles are fp\+projective;
\item any\/ $\FpProj$\+periodic object in $\sK$ is fp\+projective;
\item the cotorsion pair $(\sK_\proj^\fp$, $\sK_\inj^\fp)$ is
hereditary in\/~$\sK$;
\item the category\/ $\sK$ is locally coherent.
\end{enumerate}
\end{cor}

\begin{proof}
 (1)\,$\Longleftrightarrow$\,(2) is
essentially~\cite[Proposition~1]{EFI} (cf.\ the discussion in
Section~\ref{introd-module-periodicity-theorems-subsecn} of
the introduction).
 The implication (1)\,$\Longrightarrow$\,(2) was already explained
in the proof of Corollary~\ref{lfp-category-weakly-fp-periodicity-cor}.
 To prove (2)\,$\Longrightarrow$\,(1), all one needs to do it so
chop up a given acyclic complex of fp\+projectives into short
exact sequence pieces and apply~(2) to the coproduct of the resulting
short exact sequences.

 (3)\,$\Longleftrightarrow$\,(4) is
Corollary~\ref{two-fp-cotors-pairs-agree-equiv-loc-coherent}\,%
(1)\,$\Leftrightarrow$\,(4).

 The implication (3)\,$\Longrightarrow$\,(1) holds
by Theorem~\ref{lfp-category-cycles-in-fp-projective-thm} and
Corollary~\ref{two-fp-cotors-pairs-agree-equiv-loc-coherent}\,%
(1)\,$\Rightarrow$\,(2).

 To prove (1)\,$\Longrightarrow$\,(3), consider a short exact sequence
$0\rarrow Q\rarrow S\rarrow T\rarrow0$ in $\sK$ with fp\+projective
objects $S$ and $T$.
 Put $P_1=S$ and $P_0=T$.
 
 Any object of $\sK$ is an epimorphic image of an fp\+projective object
(in fact, any object of $\sK$ is even a pure epimorphic image of
a pure-projective object).
 So the object $Q$ has an fp\+projective resolution
$\dotsb\rarrow P_4\rarrow P_3\rarrow P_2\rarrow Q\rarrow0$.
 Then $\dotsb\rarrow P_3\rarrow P_2\rarrow P_1\rarrow P_0\rarrow0$ is
an acyclic complex $P_\bu$ of fp\+projective objects in~$\sK$.
 Among the objects of cocycles of the acyclic complex $P_\bu$,
there is the object~$Q$.
 Thus (1)~implies that $Q$ is fp\+projective.
\end{proof}

 For reader's convenience, let us explicitly formulate our results in
the case of module categories.

\begin{cor} \label{modules-cycles-in-fp-projective-cor}
 Let $R$ be an associative ring, and let $P^\bu$ be an acyclic complex
of $R$\+modules whose terms $P^n$ are fp\+projective $R$\+modules.
 Denote by $Z^n\in\Modr R$ the modules of cocycles of the acyclic
complex~$P^\bu$.
 Then all the $R$\+modules $Z^n$ are weakly fp\+projective.
\end{cor}

\begin{cor} \label{modules-hot-hom-fp-proj-fp-inj-cor}
 Let $R$ be an associative ring, let $P^\bu$ be a complex of
fp\+projective right $R$\+modules, and let $J^\bu$ be an acyclic complex
of fp\+injective right $R$\+modules with fp\+injective modules of
cocycles.
 Then any morphism of complexes $P^\bu\rarrow J^\bu$ is homotopic
to zero. 
\end{cor}

\begin{cor} \label{modules-weakly-fp-periodicity-cor}
 Let $R$ be an associative ring, and let\/ $\FpProj$ denote the class
of all fp\+projective $R$\+modules.
 Then any\/ $\FpProj$\+periodic $R$\+module is weakly fp\+projective.
\end{cor}

\begin{proof}[Proof of
Corollaries~\ref{modules-cycles-in-fp-projective-cor}\+-%
\ref{modules-weakly-fp-periodicity-cor}]
 Apply Theorems~\ref{lfp-category-cycles-in-fp-projective-thm}\+-%
\ref{lfp-category-hot-hom-fp-proj-fp-inj-thm} and
Corollary~\ref{lfp-category-weakly-fp-periodicity-cor} to the module 
category $\sK=\Modr R$.
\end{proof}

\begin{cor} \label{modules-cocycles-fp-proj-iff-coherent-cor}
 For any associative ring $R$, the following conditions are equivalent:
\begin{enumerate}
\item in any acyclic complex of right $R$\+modules with fp\+projective
terms, the modules of cocycles are fp\+projective;
\item any\/ $\FpProj$\+periodic right $R$\+module is fp\+projective;
\item the ring $R$ is right coherent.
\end{enumerate}
\end{cor}

\begin{proof}
 The nontrivial implications (3)\,$\Longrightarrow$\,(2) and
(3)\,$\Longrightarrow$\,(1) are due to \v Saroch and
\v St\!'ov\'\i\v cek~\cite[Example~4.3]{SarSt}.
 The whole corollary can be also obtained by applying our
Corollary~\ref{lfp-category-cocycles-fp-proj-iff-loc-coherent-cor}
to $\sK=\Modr R$.
\end{proof}

 For a simple counterexample of a non-fp-projective $\Proj$\+periodic
module over a noncoherent ring, see
Example~\ref{noncoherent-counterex} below.

 Notice that the assertion of
Corollary~\ref{modules-hot-hom-fp-proj-fp-inj-cor} for right coherent
rings $R$ is also covered by the discussion of \v Saroch and
\v St\!'ov\'\i\v cek; see~\cite[second paragraph of Example~4.3]{SarSt}.
 Our approach provides a generalization to arbitrary rings~$R$.

 Over an arbitrary ring $R$, the particular case of
Corollary~\ref{modules-cycles-in-fp-projective-cor} for acyclic
complexes of \emph{pure-projective} modules $P^\bu$ was obtained
by Emmanouil and Kaperonis in~\cite[Lemma~4.5(iii)
or Corollary~4.9(i)]{EK}.
 The particular case of
Corollary~\ref{modules-hot-hom-fp-proj-fp-inj-cor} for complexes of
\emph{pure-projective} modules $P^\bu$ and acyclic complexes of
\emph{strongly} fp\+injective modules $J^\bu$ with \emph{strongly}
fp\+injective modules of cocycles can be found
in~\cite[Lemma~4.5(ii)]{EK}.

\begin{rem}
 The results of this section admit a rather straightforward extension
to higher cardinalities.
 Given a regular cardinal~$\kappa$ and a locally $\kappa$\+presentable
Grothendieck category $\sK$, one defines \emph{$\kappa$\+p\+injective},
\emph{$\kappa$\+p\+projective}, \emph{strongly
$\kappa$\+p\+in\-jec\-tive}, and \emph{weakly $\kappa$\+p\+projective}
objects similarly to the definitions in Section~\ref{loc-fin-pres-secn},
using $\kappa$\+presentable objects $T$ instead of finitely
presentable ones.

 Suitable analogues of Lemmas~\ref{absolutely-pure}
and~\ref{pure-exact-structure-as-flat-modules} hold in this context,
with purity replaced by $\kappa$\+purity and $\cR$ denoting a small
additive category equivalent to the full subcategory of
$\kappa$\+presentable objects in~$\sK$.
 There is only a set of isomorphism classes of $\kappa$\+presentable
objects in $\sK$ by~\cite[Remark~1.19]{AR}.
 The category of flat $\cR$\+modules should be replaced with its full
subcategory of $\kappa$\+flat $\cR$\+modules (in the sense
of~\cite[Section~6]{Pper} in the $\kappa$\+version of
Lemma~\ref{pure-exact-structure-as-flat-modules}.
 The point is that every object of $\sK$ is a $\kappa$\+filtered
direct limit of $\kappa$\+presentable objects
by~\cite[Proposition~7.15]{JL} or~\cite[Definition~1.17
and Theorem~1.20]{AR}.
 Furthermore, the $\kappa$\+analogues of the lemmas and corollaries
of Section~\ref{hill-lemma-secn} also hold.
 So the $\kappa$\+versions of
Theorems~\ref{lfp-category-cycles-in-fp-projective-thm}\+-%
\ref{stovicek-pure-neeman} and
Corollaries~\ref{lfp-category-weakly-fp-periodicity-cor}\+-%
\ref{lfp-category-cocycles-fp-proj-iff-loc-coherent-cor} can be deduced
similarly to the proofs above.

 Generalizing the results of this section to locally
$\kappa$\+presentable (\emph{not} necessarily Grothendieck) abelian
categories would be a harder task, as the Hill lemma and its corollaries
have been only proved for Grothendieck categories
in the paper~\cite{Sto0}.
\end{rem}

\Section{Application to Derived Categories}
\label{derived-categories-secn}

 The aim of this section is to show that, for a locally coherent
abelian category $\sK$, the (unbounded) derived category $\sD(\sK)$ is
equivalent to the derived category of the exact category of
fp\+projective objects in~$\sK$,
$$
 \sD(\sK_\proj^\fp)\simeq\sD(\sK).
$$
 Here the exact category structure on $\sK_\proj^\fp$ is inherited
from the abelian exact structure of the ambient abelian category $\sK$
(notice that the full subcategory of fp\+projective objects
$\sK_\proj^\fp$ is closed under extensions in~$\sK$).

 We start with a general lemma applicable to exact categories~$\sK$.
 Denote by $\sD(\sK)$ the derived category of an exact category~$\sK$.
 So $\sD(\sK)$ is the triangulated Verdier quotient category
$\sD(\sK)=\Hot(\sK)/\Ac(\sK)$, where $\Ac(\sK)\subset\Hot(\sK)$ is
the triangulated subcategory of acyclic complexes.

\begin{lem} \label{subcategory-localization-lemma}
 Let\/ $\sK$ be an idempotent-complete exact category and\/
$\sA\subset\sK$ be a full additive subcategory.
 Assume that for any complex $K^\bu$ in\/ $\sK$ there exists
a complex $A^\bu$ in\/ $\sA$ together with a morphism of complexes
$A^\bu\rarrow K^\bu$ which is a quasi-isomorphism of complexes
in\/~$\sK$.
 Then the inclusion of additive categories\/ $\sA\rarrow\sK$ induces
a triangulated equivalence of Verdier quotient categories
$$
 \frac{\Hot(\sA)}{\Hot(\sA)\cap\Ac(\sK)}\overset\simeq\lrarrow
 \frac{\Hot(\sK)}{\Ac(\sK)}=\sD(\sK).
$$
\end{lem}

\begin{proof}
 This is a particular case of \cite[Corollary~7.2.2 or
Proposition~10.2.7(ii)]{KS} or~\cite[Lemma~1.6(a)]{Pkoszul}.
\end{proof}

 We start with the special case of a module category $\sK=\Modr R$
(for a right coherent ring~$R$).
 The following proposition is more general.

\begin{prop} \label{loc-pres-enough-proj-abelian-category-localization}
 Let\/ $\sK$ be a locally presentable abelian category with enough
projective objects, and let\/ $\sA\subset\sK$ be a full additive
subcategory containing all the projective objects of\/~$\sK$.
 Then the inclusion of additive categories\/ $\sA\rarrow\sK$ induces
a triangulated equivalence of Verdier quotient categories
$$
 \frac{\Hot(\sA)}{\Hot(\sA)\cap\Ac(\sK)}
 \overset\simeq\lrarrow\sD(\sK).
$$
\end{prop}

\begin{proof}
 Here the argument is that the assumption of
Lemma~\ref{subcategory-localization-lemma} can be satisfied by choosing
$A^\bu$ to be a suitable complex of projective objects in~$\sK$.
 There are even many ways to do so: e.~g., one can choose $A^\bu$ to be
a homotopy projective complex of projective objects, as there are
enough such complexes in any locally presentable abelian category with
enough projective objects~\cite[Corollary~6.7]{PS4}.
 Alternatively, choosing $A^\bu$ as an arbitrary complex of projectives,
one can make the cone of the morphism $A^\bu\rarrow K^\bu$ not just
an acyclic, but a \emph{contraacyclic complex in the sense of Becker},
which is a stronger property~\cite[Corollary~7.4]{PS4}.
\end{proof}

\begin{cor}
 Let $R$ be a right coherent ring, and let\/ $\Modrprojfp R$ denote
the full subcategory of fp\+projective modules in\/ $\Modr R$, endowed
with the inherited exact category structure.
 Then the inclusion of exact/abelian categories\/ $\Modrprojfp R
\rarrow\Modr R$ induces an equivalence of their unbounded derived
categories,
$$
 \sD(\Modrprojfp R)\overset\simeq\lrarrow\sD(\Modr R).
$$
\end{cor}

\begin{proof}
 Compare
Proposition~\ref{loc-pres-enough-proj-abelian-category-localization}
with Corollary~\ref{modules-cycles-in-fp-projective-cor}
or~\ref{modules-cocycles-fp-proj-iff-coherent-cor}.
\end{proof}

 Now we pass to the general case of a locally coherent category~$\sK$
(which need \emph{not} have enough projectives).
 Once again, the following proposition is even more general.

\begin{prop} \label{locally-fin-pres-abelian-fp-proj-localization}
 Let\/ $\sK$ be a locally finitely presentable abelian category and\/
$\sK_\proj^\fp\subset\sK$ be its full subcategory of fp\+projective
objects.
 Then the inclusion of additive categories\/ $\sK_\proj^\fp
\rarrow\sK$ induces a triangulated equivalence of Verdier quotient
categories
$$
 \frac{\Hot(\sK_\proj^\fp)}{\Hot(\sK_\proj^\fp)\cap\Ac(\sK)}
 \overset\simeq\lrarrow\sD(\sK).
$$
\end{prop}

\begin{proof}
 Once again, the argument is that the assumption of
Lemma~\ref{subcategory-localization-lemma} is satisfied.
 Similarly to the proof of
Proposition~\ref{loc-pres-enough-proj-abelian-category-localization},
there are several possible constructions.

 One approach is to choose $A^\bu$ as a complex of pure-projective
objects in $\sK$ and the morphism $A^\bu\rarrow K^\bu$ as a pure
quasi-isomorphism (i.~e., a morphism of complexes with pure acyclic
cone).
 To see that this can be done, interpret $\sK$ as the full subcategory
of flat modules in $\Modr\cR$, as per
Lemma~\ref{pure-exact-structure-as-flat-modules}.
 Then the pure-projective objects of $\sK$ correspond to the projective
objects of $\Modr\cR$.
 The point is that there are enough complexes of projective
modules/objects in the exact category of flat modules $\Modrfl\cR$:
any complex in $\Modrfl\cR$ is quasi-isomorphic, \emph{as a complex
in\/ $\Modrfl\cR$}, to a complex of projectives.
 This is another result of Neeman's paper~\cite{Neem};
see~\cite[Proposition~8.1 and Theorem~8.6]{Neem}.

 Alternatively, one can use a suitable version of the construction
of homotopy projective resolutions in~\cite[Section~3.A]{Spal}.
 This allows to produce a (nonpure) quasi-isomorphism $A^\bu
\rarrow K^\bu$ in $\sK$, where $A^\bu$ belongs to the minimal full
triangulated subcategory of $\Hot(\sK)$ containing the one-term
complexes formed from fp\+projective objects (or even from
pure-projective objects) and closed under countable coproducts.
 It is enough to know that the complexes in $\sK$ admit canonical
truncations, every object is a quotient object of an fp\+projective
(or even of a pure-projective) object, and countable coproducts are
exact in\/~$\sK$.
\end{proof}

\begin{cor}
 Let\/ $\sK$ be a locally coherent abelian category, and let\/
$\sK_\proj^\fp$ denote the full subcategory of fp\+projective objects
in\/ $\sK$, endowed with the inherited exact category structure.
 Then the inclusion of exact/abelian categories\/ $\sK_\proj^\fp
\rarrow\sK$ induces an equivalence of their unbounded derived
categories,
$$
 \sD(\sK_\proj^\fp)\overset\simeq\lrarrow\sD(\sK).
$$
\end{cor}

\begin{proof}
 Compare Proposition~\ref{locally-fin-pres-abelian-fp-proj-localization}
with Theorem~\ref{lfp-category-cycles-in-fp-projective-thm} or
Corollary~\ref{lfp-category-cocycles-fp-proj-iff-loc-coherent-cor}.
\end{proof}

\Section{Failure of Non-Pure Pure-Projective Periodicity}
\label{counterexamples-secn}

 Simson's theorem~\cite[Theorem~1.3 or~4.4]{Sim}
(Theorem~\ref{simson-pure-proj-periodic} in the introduction) tells
that any pure $\PProj$\+periodic module is pure-projective.
 Our Corollary~\ref{modules-weakly-fp-periodicity-cor} tells that any
$\FpProj$\+periodic module is weakly fp\+projective.
 The aim of this section is to present a variety of counterexamples
showing that a (non-pure) $\PProj$\+periodic module need \emph{not} be
pure-projective.

 Acyclic complexes of pure-projective modules were considered
in the papers~\cite{Gil2,EK}.
 Our main example is a four-term exact sequence of
pure-projective modules (over the Kronecker algebra, or over
a finite-dimensional commutative algebra, or over the algebra of
polynomials in two variables over a field) with a non-pure-projective
middle module of cocycles.
 In fact, even a $\Proj$\+periodic module over a finite-dimensional
commutative algebra over a field need not be pure-projective, as
we will see.

 We start with a simple noncoherent example.

\begin{ex} \label{noncoherent-counterex}
 The following example shows that a $\Proj$\+periodic module over
a noncoherent ring need not be fp\+projective (though it is of course
weakly fp\+projective by
Corollary~\ref{modules-weakly-fp-periodicity-cor}).
 We recall that $\Proj$ denotes the class of all projective
$R$\+modules.

 Let $V$ be a vector space of infinite countable dimension over
a field~$k$ and $R=k\oplus V$ be the trivial extension algebra,
where the basis vector of $k$ is a unit in $R$, while
the multiplication on~$V$ is zero.
 Let us say that an $R$\+module is \emph{trivial} if $V$ acts by zero
in it.
 Then we have a short exact sequence of $R$\+modules $0\rarrow V
\rarrow R\rarrow k\rarrow0$, where $R$ is the free $R$\+module with
one generator, while $V$ and~$k$ are endowed with trivial $R$\+module
structures.
 Taking the direct sum of a countable number of copies of this short
exact sequence, we obtain a short exact sequence of $R$\+modules
$$
 0\lrarrow V^{(\aleph_0)}\lrarrow R^{(\aleph_0)}\lrarrow
 k^{(\aleph_0)}\lrarrow0.
$$
 As the trivial $R$\+modules $V^{(\aleph_0)}$ and $k^{(\aleph_0)}$ are
isomorphic, we see that they are $\Proj$\+periodic over~$R$.
 Still the trivial $R$\+module~$k$ is \emph{not} fp\+projective,
since it is finitely generated but not finitely presented
(see Corollary~\ref{fp-projective-fin-generated-cor}(a)).
 Consequently, the trivial $R$\+module $V^{(\aleph_0)}\simeq
k^{(\aleph_0)}$ is \emph{not} fp\+projective, either (as the class
of all fp\+projective modules is closed under direct summands).
\end{ex}

 In the rest of this section, most of our counterexamples are based on
the following example of a non-pure-projective module.
 Denote by $K$ the Kronecker algebra over a field~$k$, described
explicitly as follows.
 A basis in $K$ as a $k$\+vector space consists of four vectors~$e_0$,
$e_1$, $x$, and~$y$.
 The multiplication is given by the rules $e_0^2=e_0$, \,$e_1^2=e_1$,
$e_0x=x$, \,$e_0y=y$, \,$xe_1=x$, \,$ye_1=y$, all the other products
of basis vectors are zero.
 The unit element is $1=e_0+e_1\in K$.

 In the discussion below, the action of the idempotent elements $e_0$
and~$e_1$ is an extra.
 It is a piece of additional information, which allows us to give
an example of a non-pure-projective $\PProj$\+periodic module over
a hereditary finite-dimensional algebra $K$ in
Example~\ref{example-Kronecker-algebra}.
 A reader more interested in commutative algebra counterexamples
(such as Examples~\ref{example-3dim-algebra},
\ref{example-polynomial-algebra}, and~\ref{example-4dim-frob})
will lose little by skipping all mentions of the action of these
idempotent elements in our modules.

 Let $M$ denote the following right $K$\+module.
 A basis in $M$ as a $k$\+vector space consists of vectors $v_{i,j}$,
where $i$, $j\in\boZ$ and $i+j=0$ or $i+j=1$.
 The action of $K$ is described by the rules
\begin{itemize}
\item $v_{i,j}e_0=v_{i,j}$ if $i+j=0$, and $v_{i,j}e_0=0$ if $i+j=1$;
\item $v_{i,j}e_1=0$ if $i+j=0$, and $v_{i,j}e_1=v_{i,j}$ if $i+j=1$;
\item $v_{i,j}x=v_{i+1,j}$ if $i+j=0$, and $v_{i,j}x=0$ if $i+j=1$;
\item $v_{i,j}y=v_{i,j+1}$ if $i+j=0$, and $v_{i,j}y=0$ if $i+j=1$.
\end{itemize}

\begin{lem} \label{not-pure-projective}
 Let $R$ be an associative ring and $R\rarrow K$ be ring homomorphism
whose image contains the elements $x$ and $y\in K$.
 Let us view $M$ as a right $R$\+module via the restriction of scalars.
 Then the $R$\+module $M$ is \emph{not} pure-projective.
\end{lem}

\begin{proof}
 For every integer $n\ge0$, denote by $M_n\subset M$ the $k$\+vector
subspace spanned by the basis vectors $v_{i,j}$, where $-n\le i\le n+1$
and $-n\le j\le n+1$ (while $i+j=0$ or~$1$, of course).
 So $M_n$ is a finite-dimensional vector subspace of dimension $4n+3$
in~$M$.
 Clearly, $M_n$ is a $K$\+submodule in~$M$.
 One has $M=\varinjlim_{n\ge0}M_n$, so there is a pure exact sequence
of right $K$\+modules (and also of right $R$\+modules)
\begin{equation} \label{module-telescope}
 0\lrarrow\bigoplus\nolimits_{n=0}^\infty M_n\lrarrow
 \bigoplus\nolimits_{n=0}^\infty M_n\lrarrow M\lrarrow0.
\end{equation}
 In order to show that $M$ is not a pure-projective $R$\+module, it
suffices to check that the short exact sequence of
$R$\+modules~\eqref{module-telescope} is not split.

 Put $S=\bigoplus_{n=0}^\infty M_n$ and $S_m=\bigoplus_{n=0}^m M_n
\subset S$ for $m\ge0$.
 For the sake of contradition, assume that $s\:M\rarrow S$ is
an $R$\+linear splitting of~\eqref{module-telescope}.
 Then there exists an integer $m\ge0$ such that $s(v_{0,0})\in S_m$.

 Arguing by induction, we will show that $s(v_{-n,n})\in
Sx+Sy+S_m$ and similarly $s(v_{n,-n})\in Sx+Sy+S_m\subset S$
for all $n\ge0$.
 For $n>m$, this will clearly contradict the assumption that
$s$~is a section of~\eqref{module-telescope} (notice that
$Mx=My=Me_1$ is a submodule in $M$ \emph{not} containing
the basis vectors $v_{-n,n}$ and~$v_{n,-n}$).

 The key observation is that both the maps
$$
 M_n/(M_nx+M_ny)\overset x\lrarrow M_n
 \quad\text{and}\quad
 M_n/(M_nx+M_ny)\overset y\lrarrow M_n
$$
are injective for all $n\ge0$.
 By the induction assumption, we have $s(v_{-n+1,n-1})\in Sx+Sy+S_m$
for some $n\ge1$.
 Hence $s(v_{-n+1,n})=s(v_{-n+1,n-1})y\in S_m$ and therefore
$s(v_{-n,n})x=s(v_{-n+1,n})\in S_m$.
 Since the map
$$
 S/(Sx+Sy+S_m)\overset x\lrarrow S/S_m
$$
is injective, it follows that $s(v_{-n,n})\in Sx+Sy+S_m$.
\end{proof}

 Now let us construct the promised four-term exact sequence.
 For every $i\in\boZ$, let $L_i\subset M$ be the $k$\+vector subspace
spanned by the three basis vectors $v_{-i,i}$,
$v_{-i+1,i}$,~$v_{-i,i+1}$.
 Clearly, $L_i$ is a $3$\+dimensional $K$\+submodule in $M$ and
$M=\sum_{i\in\boZ}L_i$.
 (In fact, one has $M_n=\sum_{i=-n}^nL_i$.)
 All the $K$\+modules $L_i$ are isomorphic to each other, so we can
put $L=L_i$.
 The kernel of the morphism $\bigoplus_{i\in\boZ} L_i\rarrow M$ is
the direct sum of a countable number of copies of the one-dimensional
$K$\+module $E=k$ with $Ex=Ey=Ee_0=0$ and $e_1$~acting in $E$ by
the identity map.

 For every $i\in\boZ$, let $Q_i\subset M$ be the $k$\+vector subspace
spanned by \emph{all} the basis vectors \emph{except} $v_{-i,i}$,
$v_{-i,i+1}$, and $v_{-i-1,i+1}$.
 Then $Q_i$ is a $K$\+submodule in $M$ with a $3$\+dimensional
quotient module $N_i=M/Q_i$.
 All the $K$\+modules $N_i$ are isomorphic to each other, so we can
put $N=N_i$.
 Furthermore, for any element $w\in M$ one has $w\in Q_i$ for all
but a finite set of integers~$i$.
 So there is a natural injective $K$\+module morphism $M\rarrow
\bigoplus_{i\in\boZ} N_i$.
 The cokernel of this morphism is isomorphic to the direct sum of
a countable number of copies of the one-dimensional $K$\+module $F=k$
with $Fx=Fy=Fe_1=0$ and $e_0$~acting in $F$ by the identity map.
 
 Thus we obtain a four-term exact sequence of $K$\+modules
\begin{equation} \label{four-term-sequence}
 0\lrarrow E^{(\aleph_0)}\lrarrow L^{(\aleph_0)}
 \lrarrow N^{(\aleph_0)}\lrarrow F^{(\aleph_0)}\lrarrow0
\end{equation}
with the middle module of cocycles equal to~$M$.
 The $K$\+modules $E$ and $F$ are one-dimensional, while
the $K$\+modules $L$ and $N$ are three-dimensional over~$k$.

\begin{ex} \label{example-Kronecker-algebra}
 The four-term exact sequence~\eqref{four-term-sequence} is
a (finite) acyclic complex of pure-projective $K$\+modules whose
middle module of cocycles $M$ is \emph{not} pure-projective
by Lemma~\ref{not-pure-projective}.
 Thus a $\PProj$\+periodic module over the hereditary finite-dimensional
algebra $K$ need \emph{not} be pure-projective.
\end{ex}

\begin{ex} \label{example-3dim-algebra}
 Let $R$ be the unital $k$\+subalgebra in $K$ spanned by the elements
$x$ and~$y$.
 So $R$ is a $3$\+dimensional commutative $k$\+algebra isomorphic to
$k[x,y]/(x^2,xy,y^2)$.
 Viewed as $R$\+modules, $L$ is a free $R$\+module with one generator,
$L\simeq R$, while $N$ is a cofree $R$\+module with one cogenerator,
$N\simeq R^*=\Hom_k(R,k)$ (so $N$ is an injective $R$\+module).
 The $R$\+modules $E$ and $F$ are isomorphic, of course (in fact,
there is a unique simple $R$\+module~$k$).

 So we have a four-term exact sequence of $R$\+modules
\begin{equation} \label{commutative-four-term-sequence}
 0\lrarrow k^{(\aleph_0)}\lrarrow R^{(\aleph_0)}
 \lrarrow R^*{}^{(\aleph_0)}\lrarrow k^{(\aleph_0)}\lrarrow0
\end{equation}
whose middle module of cocycles $M$ is \emph{not} pure-projective
by Lemma~\ref{not-pure-projective}.
 Thus a $\PProj$\+periodic module over the finite-dimensional
commutative algebra $R$ need \emph{not} be pure-projective.
\end{ex}

\begin{ex} \label{example-polynomial-algebra}
 Let $R'=k[x,y]$ be the commutative algebra of polynomials in two
variables over a field~$k$.
 Taking the restriction of scalars with respect to the obvious
surjective $k$\+algebra morphism $R'\rarrow R$, with $R$ as in
Example~\ref{example-3dim-algebra}, one can
view~\eqref{commutative-four-term-sequence} as a four-term exact
sequence of pure-projective $R'$\+modules.
 (Notice that the algebra $R'$ is Noetherian, so any finitely
generated $R'$\+module is finitely presented.)

 The $R'$\+module $M$ is \emph{not} pure-projective
by Lemma~\ref{not-pure-projective}.
 Thus a $\PProj$\+periodic module over the regular finitely generated
commutative $k$\+algebra $R'=k[x,y]$ need \emph{not} be pure-projective.
\end{ex}

\begin{ex}\label{example-4dim-frob}
 Put $R''=k[x,y]/(x^2,y^2)$; so $R''$ is a $4$\+dimensional Frobenius
commutative $k$\+algebra.
 As such, any $R''$\+module has a double-sided projective-injective
$R''$\+module resolution.
 In other words, any $R''$\+module can be obtained as the module of
cocycles (in some particular cohomological degree) of an unbounded
acyclic complex of projective-injective $R''$\+modules.

 Taking the restriction of scalars with respect to the obvious
surjective $k$\+algebra morphism $R''\rarrow R$, one can view $M$
as an $R''$\+module.
 The $R''$\+module $M$ is \emph{not} pure-projective by
Lemma~\ref{not-pure-projective}.
 Still, it can be obtained as a module of cocycles in an acyclic
complex of projective $R''$\+modules.
 Thus a $\Proj$\+periodic module over the finite-dimensional
commutative $k$\+algebra $R''$ need \emph{not} be pure-projective.
\end{ex}

\begin{ex}\label{example-frob}
More generally, let $S$ be any quasi-Frobenius ring which is not right pure semisimple. The same argument as in Example~\ref{example-4dim-frob} applies and shows that every $S$-module is a cocycle in an unbounded acyclic complex of projective-injective $S$\+modules. The fact that $S$ is not right pure semisimple amounts to the existence of a right $S$-module which is not pure-projective. It follows that $S$ admits a right $\Proj$\+periodic module which is not pure-projective. 

A commutative ring is pure semisimple if and only if it is an artinian principal ideal ring, see \cite[\S 1.2]{W}. Also, a commutative artinian local ring is a principal ideal ring if and only if it is a hypersurface, this follows directly from \cite[Corollary 11]{H}. Therefore, any commutative artinian local Gorenstein ring which is not a hypersurface can play the role of $S$, this includes the ring of Example~\ref{example-4dim-frob}.
\end{ex}

\begin{ex}
Even more generally, let $S$ be a ring which admits a right $S$-module $M$ which is Gorenstein projective but not pure-projective (see e.g. \cite[\S 3]{SarSt} for the definition of a Gorenstein projective module). By definition, $M$ is a cocycle in an acyclic complex of projective right $S$-modules, and therefore there is a $\Proj$\+periodic right $S$-module which is not pure-projective.

A source of such rings $S$ can be obtained as follows. Let $S$ be a commutative noetherian complete local Gorenstein ring which is not regular. A result of Beligiannis \cite[Theorem 4.20]{B} asserts that all Gorenstein projective $S$-modules are pure-projective if and only if $S$ is CM-finite, the latter means that the category of Gorenstein projective $S$-modules is of finite representation type. Note that the proof of \cite[Theorem 4.20]{B} explicitly uses the fact that an $S$-module is pure-projective precisely if it is isomorphic to a direct sum of finitely presented $S$-modules. By \cite[Corollary 4.21]{B}, the CM-finiteness implies that $S$ is a simple hypersurface in the sense of \cite{BGS}. In particular, any commutative noetherian complete local Gorenstein ring which is not a hypersurface can play the role of $S$.
\end{ex}

\begin{ex}
 This is an example of non-pure-projective $\PProj$\+periodic module
over a valuation domain.
 Over a Pr\"ufer domain, the class of fp\+injective modules coincides
with the class $\sD$ of divisible
modules~\cite[Proposition~IX.3.4]{FS}.
 If $\sP_1$ is the class of modules of projective dimension at most one,
then for every commutative domain, $(\sP_1, \sD)$ is a complete
hereditary  cotorsion pair~\cite[Theorem~7.2]{BH}.
 Thus, for Pr\"ufer domains the class $\sP_1$ coincides with the class
of fp\+projective modules.

 Let $R$ be a valuation domain with value group the abelian group
$\boZ\oplus\boZ$ with the anti-lexicographic order.
 The maximal ideal of $R$ is principal generated by an element~$r_0$
with value $(1,0)$ and $\bigcap_{n\geq 0} r_0^nR$ is a prime ideal
$\mathfrak p$ generated by elements $s_n$, \,$n\geq 0$ with
value $(-n,1)$.
 One can choose the elements~$s_n$ so that $s_{n+1}r_0=s_n$ for every
$n\geq 0$.

 Let $P$ be the pure-projective module $\bigoplus_{n\geq 0} R/s_nR$
and let $e_n=1+s_nR$ be the basis elements in~$P$.
 Consider the submodule $M$ of $P$ generated by the elements 
$x_n=(e_0+e_1+\dotsb+e_{n-1})s_n$ for every $n\geq 1$.

 For every $n\geq 1$, one has $x_{n+1}r_0=x_n$, thus
$M=\bigcup_{n\geq 1}x_nR$ and $M$ is isomorphic to
$\bigcup_{n\geq 1}\frac{r_0^{-n}R}R\subset\frac QR$, where $Q$ is
the quotient field of $R$. 
 This shows that $M$ has projective dimension one, hence it is
fp\+projective, but it is not pure-projective. Indeed, $M$ is uniserial
(see~\cite[p.~24 in Section~I.4]{FS} or~\cite[Chapter~X]{FS}), and
the only modules over $R$ which are both uniserial and pure-projective
are the cyclically presented ones.
 The latter assertion holds because any pure-projective $R$\+module is
a direct summand of a direct sum of cyclically presented ones
(see~\cite[Theorem~1]{W0} or~\cite[Theorem~V.3.3]{FS}), and any such
direct summand is itself a direct sum of cyclically presented
modules (by~\cite[Theorem I.9.8]{FS}).

 The claim is that $P/M$ is again a pure-projective module, so that  
considering a pure-projective resolution $P_\bu$ of $M$, the complex
$P_\bu\rarrow P\rarrow P/M\rarrow 0$ with pure-projective terms has $M$
as a cocycle.
 Let us explain why this is the case.
 
 For every $n\geq 1$, write $P=C_n\oplus P_n$, where
$C_n=\bigoplus_{0\leq i\leq n-1} R/s_iR$ and
$P_n=\bigoplus_{i\geq n}R/s_iR$.
 Now the quotient module $P/M$ is isomorphic to
$\varinjlim_{n\geq 1} P/x_nR$, and since $x_n\in C_n$, we have that
$P/M$ is isomorphic to $\varinjlim_{n\geq 1} C_n/x_nR$.
 Here the transition morphisms $\pi_n\: C_n/x_nR\rarrow
C_{n+1}/x_{n+1}R$ are given by the obvious rule $\pi_n(e_i+x_n R)=
e_i+x_{n+1}R$ for all $0\le i\le n-1$.

 Put $f_1=e_0$, $f_2=e_0+e_1$, \dots, $f_j=\sum_{i=0}^{j-1}e_i$
for all $j\in\boZ$.
 The $R$\+module $C_n$ is generated by the elements $e_0$,~\dots,
$e_{n-1}$ with the relations $e_is_i=0$ for all $0\le i\le n-1$.
 Changing the basis $\{e_0, e_1,\dotsc,e_{n-1}\}$ of $C_n$ 
to the basis $\{f_1,f_2,\dotsc,f_n\}$, we see that $C_n$ is
generated by the elements $f_1$,~\dots, $f_n$ subject to the relations
$f_1s_0=0$, \,$(f_2-f_1)s_1=0$, \dots, $(f_{n-1}-f_{n-2})s_{n-2}=0$,
\,$(f_n-f_{n-1})s_{n-1}=0$.

 Let us compute the $R$\+module $D_n=C_n/x_nR$.
 Put $\bar f_i=f_i+x_nR\in D_n$ for all $1\le i\le n$.
 We have $x_n=f_ns_n$; so the additional relation $\bar f_ns_n=0$ needs
to be imposed on top of the relations in $C_n$ in order to
construct~$D_n$.
 Now $\bar f_ns_n=0$ implies $\bar f_ns_{n-1}=0$, so the relation
$(\bar f_n-\bar f_{n-1})s_{n-1}=0$ can be rewritten simply as
$\bar f_{n-1}s_{n-1}=0$.
 This, in turn, implies $\bar f_{n-1}s_{n-2}=0$, so
$(\bar f_{n-1}-\bar f_{n-2})s_{n-2}=0$ means simply
$\bar f_{n-2}s_{n-2}=0$, etc.
 Proceeding in this way, we see that the $R$\+module $D_n$ is
generated by the elements~$\bar f_1$,~\dots, $\bar f_n$ subject to
the relations $\bar f_is_i=0$ for all $1\le i\le n$. 
 (At the last step, we conclude that $(\bar f_2-\bar f_1)s_1=0$ is
equivalent to $\bar f_1s_1=0$ modulo the previous relations, and
$\bar f_1s_0=0$ is redundant.)

 We have shown that $D_n=R/s_1R\oplus R/s_2R\oplus\dotsb\oplus R/s_nR$.
 Computing the morphism $\pi_n\:D_n\rarrow D_{n+1}$ in the new basis,
one finds that it is still given by the obvious rule
$\pi_n(\bar f_i)=\bar f_i$ for all $1\le i\le n$.
 So $\pi_n$~is a split monomorphism with the cokernel $R/s_{n+1}R$.
 Hence $P/M$ is isomorphic to $R/s_1R\oplus R/s_2R\oplus\dotsb
\oplus R/s_nR\oplus\dotsb$ and thus it is pure-projective.
\end{ex}


\bigskip
\begin{thebibliography}{99}
\smallskip

\bibitem{AR}
 J.~Ad\'amek, J.~Rosick\'y.
   Locally presentable  and accessible categories.
London Math.\ Society Lecture Note Series~189,
Cambridge University Press, 1994.

\bibitem{BCE}
 S.~Bazzoni, M.~Cort\'es-Izurdiaga, S.~Estrada.
   Periodic modules and acyclic complexes.
\textit{Algebras and Represent.\ Theory} \textbf{23}, \#5,
p.~1861--1883, 2020.  	\texttt{arXiv:1704.06672 [math.RA]}

\bibitem{BH}
 S.~Bazzoni, D.~Herbera.
   Cotorsion pairs generated by modules of bounded projective dimension.
\textit{Israel Journ.\ of Math.}\ \textbf{174}, p.~119--160, 2009.
\texttt{arXiv:0707.2026 [math.RA]}

\bibitem{B}
A.~Beligiannis.
   On algebras of finite Cohen–Macaulay type. 
\textit{Adv.\ Math.}  \textbf{226}, \#2,
p.~1973--2019, 2011.

\bibitem{BG}
 D.~J.~Benson, K.~R.~Goodearl.
   Periodic flat modules, and flat modules for finite groups.
\textit{Pacific Journ.\ of Math.}\ \textbf{196}, \#1, p.~45--67, 2000.

\bibitem{BGS}
R.-O.~Buchweitz, G.-M.~Greuel, and F.-O. Schreyer. 
   Cohen-Macaulay modules on hypersurface singularities II.
\textit{Invent. Math.} \textbf{88}, \#1,
p.~165--182, 1987.

\bibitem{CH}
 L.~W.~Christensen, H.~Holm.
   The direct limit closure of perfect complexes.
\textit{Journ.\ of Pure and Appl.\ Algebra} \textbf{219}, \#3,
p.~449--463, 2015.  \texttt{arXiv:1301.0731 [math.RA]}

\bibitem{CB}
 W.~Crawley-Boevey.
   Locally finitely presented additive categories.
\textit{Communicat.\ in Algebra} \textbf{22}, \#5, p.~1641--1674, 1994.

\bibitem{EK}
 I.~Emmanouil, I.~Kaperonis.
   On K\+absolutely pure complexes.
Available from \texttt{http://users. uoa.gr/\textasciitilde
emmanoui/research.html}

\bibitem{EFI}
 S.~Estrada, X.~Fu, A.~Iacob.
   Totally acyclic complexes.
\textit{Journ.\ of Algebra} \textbf{470}, p.~300--319, 2017.
\texttt{arXiv:1603.03850 [math.AC]}

\bibitem{FS}
 L.~Fuchs, L.~Salce.
   Modules over Non-Noetherian Domains.
Mathematical Surveys and Monographs 84, American Math.\ Society,
Providence, 2001.

\bibitem{Gil}
 J.~Gillespie.
   The flat model structure on $\mathbf{Ch}(R)$.
\textit{Trans.\ of the Amer.\ Math.\ Soc.}\ \textbf{356}, \#8,
p.~3369--3390, 2004.

\bibitem{Gil2}
 J.~Gillespie.
   The homotopy category of acyclic complexes of pure-projective
modules.
\textit{Forum Mathematicum} \textbf{35}, \#2, p.~507--521, 2023.
\texttt{arXiv:2201.05542 [math.AT]}

\bibitem{Gla}
 S.~Glaz.
   Commutative coherent rings.
\textit{Lecture Notes in Math.}\ \textbf{1371}, Springer, Berlin,
1989.

\bibitem{GT}
 R.~G\"obel, J.~Trlifaj.
   Approximations and endomorphism algebras of modules.
Second Revised and Extended Edition.
De Gruyter Expositions in Mathematics 41,
De Gruyter, Berlin--Boston, 2012.

\bibitem{H}
T.~W.~Hungerford.
    On the structure of principal ideal rings.
\textit{Pacific J. Math.}\ \textbf{25},
p.~543--547, 1968.

\bibitem{JL}
 C.~U.~Jensen, H.~Lenzing.
   Model-theoretic algebra (with particular emphasis on fields,
rings, modules).
Algebra, Logic, and Applications, 2.
 Gordon and Breach Science Publishers, New York, 1989.

\bibitem{KS}
 M.~Kashiwara, P.~Schapira.
   Categories and sheaves.
Grundlehren der mathematischen Wissenschaften, 332,
Springer, 2006.

\bibitem{Kra}
 H.~Krause.
   Functors on locally finitely presented additive categories.
\textit{Colloquium Math.} \textbf{75}, \#1, p.~105--132, 1998.

\bibitem{MD}
 L.~Mao, N.~Ding.
   Notes on $FP$\+projective modules and $FP$\+injective modules.
Advances in Ring Theory, J.~Chen, N.~Ding, H.~Marubayashi, Eds.,
Proceedings of the 4th China--Japan--Korea International
Conference, June~2004, World Scientific, 2005, p.~151--166.

\bibitem{Neem}
 A.~Neeman.
   The homotopy category of flat modules, and Grothendieck duality.
\textit{Inventiones Math.}\ \textbf{174}, \#2, p.~255--308, 2008.

\bibitem{Pkoszul}
 L.~Positselski.
   Two kinds of derived categories, Koszul duality, and
comodule-contramodule correspondence.
\textit{Memoirs of the American Math.\ Society} \textbf{212},
\#996, 2011.  vi+133~pp.  \texttt{arXiv:0905.2621 [math.CT]}

\bibitem{Pfp}
 L.~Positselski.
   Coherent rings, fp\+injective modules, dualizing complexes, and
covariant Serre--Grothendieck duality.
\textit{Selecta Math.\ (New Ser.)} \textbf{23}, \#2, p.~1279--1307,
2017.  \texttt{arXiv:1504.00700 [math.CT]}

\bibitem{Pksurv}
 L.~Positselski.
   Differential graded Koszul duality: An introductory survey.
\textit{Bulletin of the London Math.\ Society} \textbf{55}, \#4,
p.~1551--1640, 2023.  \texttt{arXiv:2207.07063 [math.CT]}

\bibitem{PR}
 L.~Positselski, J.~Rosick\'y.
   Covers, envelopes, and cotorsion theories in locally presentable
abelian categories and contramodule categories.
\textit{Journ.\ of Algebra} \textbf{483}, p.~83--128, 2017.
\texttt{arXiv:1512.08119 [math.CT]}

\bibitem{Pper}
 L.~Positselski.
   Abelian right perpendicular subcategories in module categories.
Electronic preprint \texttt{arXiv:1705.04960 [math.CT]}. 

\bibitem{PS3}
 L.~Positselski, J.~\v St\!'ov\'\i\v cek.
   Topologically semisimple and topologically perfect topological rings.
\textit{Publicacions Matem\`atiques} \textbf{66}, \#2, p.~457--540,
2022.  \texttt{arXiv:1909.12203 [math.CT]}

\bibitem{PS4}
 L.~Positselski, J.~\v St\!'ov\'\i\v cek.
   Derived, coderived, and contraderived categories of locally
presentable abelian categories.
\textit{Journ.\ of Pure and Appl.\ Algebra} \textbf{226}, \#4,
article ID~106883, 2022, 39~pp.  \texttt{arXiv:2101.10797 [math.CT]}

\bibitem{PS5}
 L.~Positselski, J.~\v St\!'ov\'\i\v cek.
   Coderived and contraderived categories of locally presentable
abelian DG\+categories.
Electronic preprint \texttt{arXiv:2210.08237 [math.CT]}.

\bibitem{Ro}
 J.-E.~Roos.
   Locally Noetherian categories and generalized strictly linearly
compact rings.  Applications.
\textit{Category Theory, Homology Theory, and their Applications, II},
Springer, Berlin, 1969, p.~197--277.

\bibitem{SaoSt}
 M.~Saor\'\i n, J.~\v St\!'ov\'\i\v cek.
   On exact categories and applications to triangulated adjoints
and model structures.
\textit{Advances in Math.}\ \textbf{228}, \#2, p.~968--1007, 2011.
\texttt{arXiv:1005.3248 [math.CT]}

\bibitem{SarSt}
 J.~\v Saroch, J.~\v St\!'ov\'\i\v cek.
   Singular compactness and definability for $\Sigma$\+cotorsion and
Gorenstein modules.
\textit{Selecta Math.\ (New Ser.)}\ \textbf{26}, \#2, Paper No.~23,
40~pp., 2020.  \texttt{arXiv:1804.09080 [math.RT]}

\bibitem{Sim}
 D.~Simson.
   Pure-periodic modules and a structure of pure-projective resolutions.
\textit{Pacific Journ.\ of Math.}\ \textbf{207}, \#1, p.~235--256, 2002.

\bibitem{Spal}
 N.~Spaltenstein.
   Resolutions of unbounded complexes.
\textit{Compositio Math.}\ \textbf{65}, \#2, p.121--154, 1988.

\bibitem{Sto0}
 J.~\v St\!'ov\'\i\v cek.
   Deconstructibility and the Hill Lemma in Grothendieck categories.
\textit{Forum Math.}\ \textbf{25}, \#1, p.~193--219, 2013.
\texttt{arXiv:1005.3251 [math.CT]}

\bibitem{Sto-ICRA}
 J.~\v St\!'ov\'\i\v cek.
   Exact model categories, approximation theory, and cohomology of
quasi-coherent sheaves.
\textit{Advances in representation theory of algebras}, p.~297--367,
EMS Ser.\ Congr.\ Rep., Eur.\ Math.\ Soc., Z\" urich, 2013.
\texttt{arXiv:1301.5206 [math.CT]}

\bibitem{Sto}
 J.~\v St\!'ov\'\i\v cek.
   On purity and applications to coderived and singularity categories.
Electronic preprint \texttt{arXiv:1412.1615 [math.CT]}.

\bibitem{Trl}
 J.~Trlifaj.
   Covers, envelopes, and cotorsion theories.
Lecture notes for the workshop ``Homological methods in module theory'',
Cortona, September~2000.  39~pp.  Available from
\texttt{http://matematika.cuni.cz/dl/trlifaj/NALG077cortona.pdf}

\bibitem{W0}
 R.~B.~Warfield.
   Decomposability of finitely presented modules.
\textit{Proc.\ of the Amer.\ Math.\ Soc.}\ \textbf{25}, \#1,
p.~167--172, 1970.

\bibitem{W}
R.~B.~Warfield.
Large modules over Artinian rings. 
\textit{Representation theory of algebras, Lecture Notes in Pure and Applied Mathematics} \textbf{37}, p.~451--463, 1978.

\end{thebibliography}
\end{document}